\documentclass{tMAM2e}
\usepackage{epstopdf}
\usepackage[bookmarksnumbered,colorlinks,bookmarks,citecolor=blue,linkcolor=blue,urlcolor=blue,breaklinks,linktocpage]{hyperref}
\usepackage[all]{hypcap}
\usepackage{mathtools}
\usepackage{thm-restate}
\usepackage{tikz,pgfplots}
\pgfplotsset{compat=1.16}
\usetikzlibrary {arrows.meta, positioning} 

\newtheorem{thm}{Theorem}[section]

\newtheorem{prop}[thm]{Proposition}

\theoremstyle{definition}

\newtheorem{defn}[thm]{Definition}
\newtheorem{examp}[thm]{Example}

\theoremstyle{remark}
\newtheorem{rem}[thm]{Remark}

\newcommand{\tB}{\text{B}}
\newcommand{\G}{\mathcal{G}}
\renewcommand{\S}{\mathcal{S}}
\newcommand{\V}{\mathcal{V}}
\newcommand{\FF}{\mathbb{F}}
\newcommand{\QQ}{\mathbb{Q}}
\newcommand{\ZZ}{\mathbb{Z}}
\newcommand{\1}{\mathbf{1}}
\newcommand{\up}{\mathord{\uparrow}}
\newcommand{\down}{\mathord{\downarrow}}

\newcommand{\cross}{\times}
\newcommand{\size}[1]{\lvert #1 \rvert}

\newcommand{\uu}{\textsuperscript}
\newcommand{\Fib}{\mathrm{Fib}}
\renewcommand{\epsilon}{\varepsilon}

\makeatletter
\newcounter{flet}[table]
\newcommand{\ftn}[1]{%
  \stepcounter{flet}%
  \phantomsection%
  \def\@currentlabel{\emph{\alph{flet}}}{\label{#1}}%
  \textsuperscript{\emph{\alph{flet}}}
}
\makeatother

\begin{document}


\title{Renaissance canons with asymmetric schemes}

\author{
\name{Evan M. O'Dorney%
  \textsuperscript{a}$^{\ast}$\thanks{$^\ast$Corresponding author. Email: eodorney@andrew.cmu.edu}
}
\affil{\textsuperscript{a}Department of Mathematics, Carnegie Mellon University, Pittsburgh, Pennsylvania, USA}
\received{}
}

\maketitle

\begin{abstract}
By a \emph{scheme} of a musical canon, we mean the time and pitch displacement of each entering voice. When the time displacements are unequal, achieving consonant sonorities is especially challenging. Using a first-species theoretical model, we quantify the flexibility of schemes that Renaissance composers used or could have used. We craft an algorithm to compute this flexibility value precisely (finding in the process that it is an algebraic integer). We find that Palestrina consistently selected some of the most flexible schemes, more so than his predecessors, but that he by no means exhausted all feasible schemes. To add support to the model, we present two new compositions within the limits of the style utilizing unexplored canonic schemes.

In the Online Supplement, we provide MIDI realizations of the musical examples and Sage code used in the numerical computations.
\end{abstract}


\begin{classcode}\textit{2010 Mathematics Subject Classification}:
  00A65; 
  68R10; 
  05C38; 
  90B15; 
  11R06  
\end{classcode}

\section{Introduction}
A \emph{canon} is a musical piece in which two or more voices perform the same melodic line, but displaced in time. In the simplest canons, the voices enter at equally spaced times and at the same pitch, forming a \emph{catch} or \emph{round} (such as the medieval ``Sumer is icumen in'' or the more modern ``Fr\`ere Jacques''). In others, the pitches of each voice are related by transposition, inversion, or the like. The use of canon presents an extra layer of challenge to the composer, above and beyond the demands of a musical style. This challenge has the form of a mathematically precise constraint deriving one voice from another. Canons (in various styles) have formed the subject matter of numerous articles in the mathematical study of music, which have even spawned mathematical problems studied in their own right (see \cite{Geenen,tilings}).


Medieval and early Renaissance composers produced a menagerie of canons (see \cite{Feininger}) including such devices as inversion, retrograde, omission of some indicated notes, augmentation, diminution, and simultaneous meters (the last found, for instance, in Ockeghem's \emph{Missa prolationum}).  By the time of Giovanni Pierluigi da Palestrina (c.~1520--1594), composers were losing interest in the more recondite kinds of canons, and there appeared other specimens which, while superficially of simple type, demand great compositional skill. Such are canons where \emph{three} or more voices sing from a single line, with pitch transposition. Palestrina's output includes a number of these, which are listed in Table \ref{tab:Pal_canons}, grouped by their canonic schemes. By a \emph{scheme} we mean the rule by which the following voice(s) of the canon are derived from the leading voice; the notation for the schemes is explained in the next section. We are here especially interested in the time displacements between the entries of successive voices; we call \emph{asymmetric} a canon in three or more parts where these time displacements are unequal.

We have found Palestrina a convenient composer to highlight here because his style has been intensively studied as representing the height of the Renaissance, and because he takes pains not to relax the rules of the style, which pertain especially to dissonance treatment, even in the face of additional demands imposed by a canonic structure \citep[p.~199] {jeppesen2012style}. In the course of the sixteenth century from Josquin to Palestrina, the rules for treating dissonance grew increasingly strict, so that writing a canon became more challenging. The plethora of written knowledge about the Palestrina style have made it an attractive aim for computerized music composition since the 1980's, and the quest for algorithms better imitating this style remains tantalizing today \citep{Lewin,Farbood,Herremans}.

Very little is known about how Renaissance canons were written.  \citet{GauldinComposition} suggests that a skeleton in first species (whole notes only) could have been written first, based on the observation that some canons in the repertoire can be analyzed as a single underlying motive whose repetition is obscured by varied ornamentation of the melodic line. The question then arises of which schemes admit a canon on some motive, or even better, on several motives that can be interlocked in varying ways. In this paper, we address this question as follows:
\begin{itemize}
  \item Using a theoretical model based on first-species reduction, we give each canonic scheme $\S$ a \emph{flexibility} value $\lambda(\S)$, a real number ranging from $0$ to $5$, such that the higher the flexibility value, the more possibilities open up to the composer using the scheme (Sections \ref{sec:types}--\ref{sec:theor}). Intuitively, the flexibility value computes the number of possibilities for each successive note of the canon, ignoring octaves and chromatic alterations.
  \item We exhaustively compute the flexibility values of canonic schemes with imitative displacements up to $8$ beats (semibreves) (Sections \ref{sec:properties}--\ref{sec:computing}). We verify that some schemes that are symmetric (i.e.~have equal time intervals) have a flexibility value much higher than average, offering a new explanation for their ubiquity in the repertoire.
  \item Some schemes have flexibility values that are exact integers, such as $1$, $2$ and $3$; these canons can be composed using the techniques of invertible counterpoint already known in the Renaissance. In general, the flexibility is an algebraic integer (the root of a polynomial with integer coefficients and leading coefficient $1$), often an irrational number of high degree, pointing to a more intricate composition process.
  \item With asymmetric canons, we find that Palestrina picks out the more flexible schemes (those with $\lambda \gtrsim 2.7$) more consistently than his predecessors, suggesting he had a better knack for devising schemes that would support interesting melodies without relaxing the dissonance treatment.
  \item However, our model suggests new schemes that should be just as flexible as those in the repertoire. To vindicate the model, we present two original compositions on new schemes in accord with the rules of counterpoint practiced in the late 16th century (Section \ref{sec:new}).
\end{itemize}


\section{Types of canons}
\label{sec:types}

\begin{table}
  \begin{tabular}{lp{2.2in}p{2.2in}}
    \multicolumn{3}{l}{\textbf{I. Stacked Canons}} \\
    Page & Work & Scheme \\ \hline
    XI, 67 & Mass \emph{Ad fugam,} Benedictus & $\{(0,0), (4,-4), (8,-8)\tB\}$ \\
    XI, 152 & Mass \emph{Papae Marcelli,} Agnus Dei II &
    $\{(0,0),(6,4),(12,8)\}$ \\
    XVII, 129ff. & Mass \emph{Sacerdotes Domini} & $\{(0,0), (3,1), (6,2)\}$ (Kyrie I, Credo, Sanctus) \\
    && $\{(0,0), (4,1), (8,2)\}$ (Christe, Kyrie II, Agnus Dei I--II, Gloria, Et unam sanctam) \\
    && $\{(0,0)\tB, (4,1), (8,2)\}$ (Pleni)\\
    && $\{(0,0), (6,-1), (12,-2)\}$ (Hosanna; time units are semibreves in $3/1$ time). \\
    \multicolumn{3}{l}{\textbf{II. Inverting Canons} (both at the 12th)} \\
    Page & Work & Scheme \\ \hline
    I, 158 & Motet ``Virgo prudentissima,'' secunda pars & $\{(0,0), (6,4), (12,-3)\}$ \\
    XIX, 34 & Mass \emph{Gi\`a f\`u chi m'hebbe cara,} Benedictus & $\{(0,0),(4,-4)\tB,(8,3)\}$ \\
    \\
    \multicolumn{3}{l}{\textbf{III. Asymmetric Canons}} \\
    Page & Work & Scheme \\ \hline
    I, 152 & Motet ``Virgo prudentissima,'' prima pars & $\{(0,0), (4,4), (10,-3)\}$ \\
    III, 123 & Motet ``Accepit Jesus calicem'' & $\{(0,0), (1,3), (5,7)\}$ \\
    XI, 53 & Mass \emph{Sine Nomine,} Agnus Dei II & $\{(0,0),(4,4),(14,6),(18,10)\}$ \\
    XI, 65 & Mass \emph{Ad fugam,} Pleni &
    $\{(0,0)\tB,(1,3),(3,7)\}$ \\
    XI, 69 & Mass \emph{Ad fugam,} Agnus Dei II &
    $\{(0,0),(6,0),(7,3)\}$ (plus a $2$-voice canon) \\
    XII, 47 & Mass \emph{Primi toni,} Agnus Dei II &
    $\{(0,0),(7,-7),(8,-4)\}$ \\
    XII, 132 & Mass \emph{Repleatur os meum,} Agnus Dei II &
    $\{(0,0),(8,7),(12,3)\}$
  \end{tabular}
  \caption{Canons of 3-in-1 or 4-in-1 type in the works of Palestrina. Volume and page numbers in the table refer to the Breitkopf \& H\"artel edition of the complete works of Palestrina, available online on IMSLP \citep{Pal}.}
  \label{tab:Pal_canons}
\end{table}

In Table \ref{tab:Pal_canons}, we have listed all canons in the works of Palestrina where at least three parts read from the same melodic line.  It will be noted that the great majority of these occur in masses, specifically in the movements Pleni, Benedictus, and Agnus Dei II. In the Pleni and Benedictus, which are middle movements within the Sanctus, it was fairly common to reduce the texture to three voices, raising the possibility of a three-voice unaccompanied canon. The Agnus Dei II, by contrast, is the grand finale of the mass setting, frequently involving added voices and an intensification of the compositional procedures in use. The three Palestrina masses using canons as a structural motif throughout (\emph{Ad fugam,} \emph{Repleatur os meum,} and \emph{Sacerdotes Domini}) all conclude with three-voice canons combined with additional voices.

Each canon has a \emph{scheme} which we notate as a set of ordered pairs $(t_i, p_i)$, meaning that a voice begins $t_i$ beats after the written timing and $p_i$ diatonic steps higher than the written pitch. In this style, the beat (\emph{tactus}) is usually a semibreve. The leading voice (\emph{dux}) sings the music as written and thus gets the ordered pair $(0,0)$. The other voices (\emph{comites}) enter later and thus have positive values of $t_i$. A negative pitch displacement $p_i$ indicates that a particular voice sings \emph{below} the written pitch. The annotation ``$\tB$'' following an ordered pair indicates that this voice is the bass of the texture. This is significant when it happens, because in Renaissance style, certain intervals (perfect and augmented fourths, diminished fifths) are treated as consonances between upper voices but dissonances when they involve the bass \citep[p.~175]{jeppesen1992counterpoint}.

As indicated in the table, the canons are of several types. The first and simplest is termed \emph{stacked canon} in \citet{GosmanStacked} and is used by many composers in the Renaissance and after. It occurs when each new voice has the same displacement in both time and pitch from the one before it; that is, the sequences of $t_i$ and of $p_i$ are both arithmetic progressions. As a special case, if all $p_i = 0$, then we have successive voices entering at the same pitch, forming a \emph{catch} as mentioned in the introduction. Palestrina did not publish any catches, but they seem to form the vast majority of canonic writing in all historical periods: some compilations include hundreds of catches, e.g.~\cite{Aldrich}. In a catch, the intervals occurring between the first two voices reappear between the second and third, and so on through the voices; in a general stacked canon, the same holds, possibly with chromatic alterations, as shown by the brackets in Figure \ref{fig:sacerdotes_pleni}.
\begin{figure}
  \includegraphics[width=\linewidth]{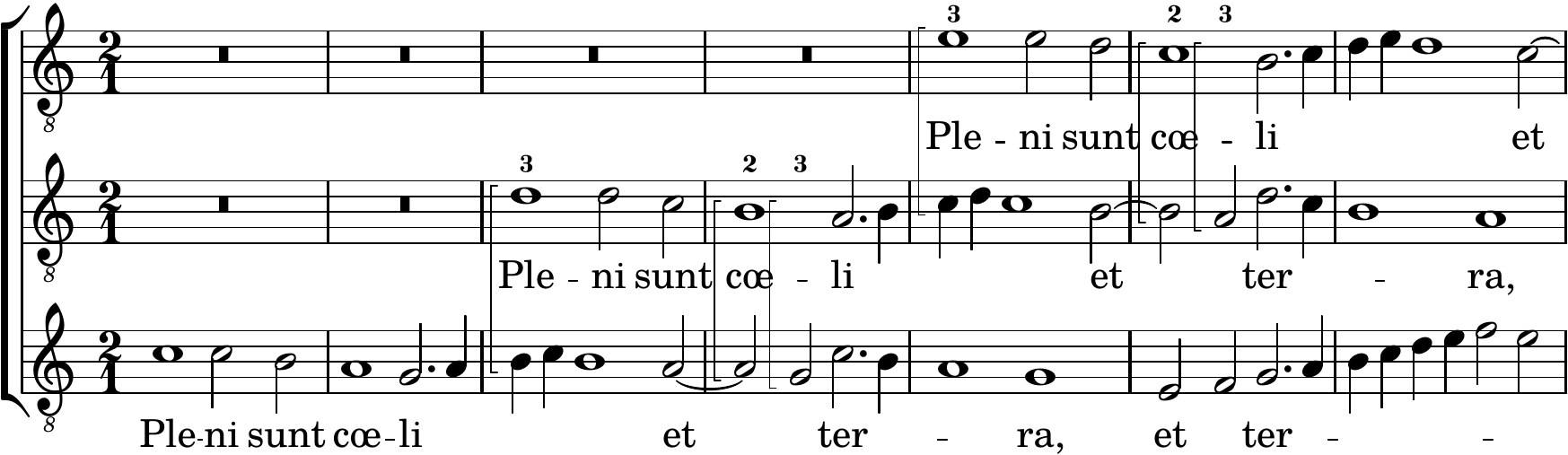}
  \caption{A stacked canon from Palestrina's mass \emph{Sacerdotes Domini}}
  \label{fig:sacerdotes_pleni}
\end{figure}

More artistically, a composer can employ equal time intervals but alter the pitch displacements so that the intervals between the first and second voices are inverted when they recur between the second and third voices. We call this disposition an \emph{inverting canon} because it requires the first two voices to be in invertible (double) counterpoint. As any counterpoint textbook will mention, invertible counterpoint at the octave, 10th, or 12th is feasible. The scheme employed by Palestrina in Figure \ref{fig:GiaFuBenedictus} causes the intervals between the two lower parts to reappear between the outer parts, but inverted at the 12th (that is, numerically subtracted from 13), as shown by the brackets. We further term this example a \emph{$\vee$-canon,} the pictorial symbol $\vee$ meaning that the bottom voice enters second; the other possibility is a $\wedge$-canon, in which the top voice enters second. We will revisit the distinction between inverting canons of $\vee$ and $\wedge$ type in Section \ref{sec:data}.

\begin{figure} 
  \includegraphics[width=\linewidth]{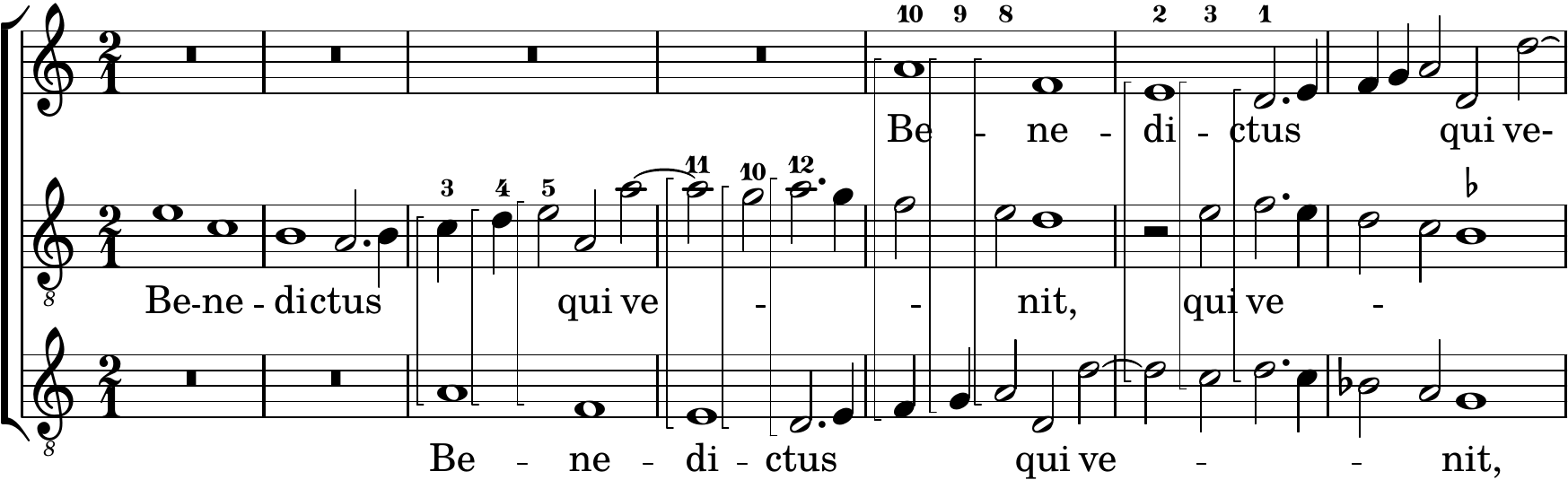}
  \caption{An inverting $\vee$-canon at the 12th from Palestrina's mass \emph{Gi\`a f\`u chi m'hebbe cara}}
  \label{fig:GiaFuBenedictus}
\end{figure}

Are equal time intervals necessary for a feasible canon? After a brief introduction to stacked and inverting canons, Gauldin's popular textbook continues: ``Opening points of canonic imitation employing \emph{both} asymmetrical intervallic and temporal relations create extreme problems as regards the continuation of the strict imitation'' \citep*[p.~115]{gauldin2013practical}. Nevertheless, numerous examples exist in the Renaissance repertoire, and Palestrina wrote no less than seven canons of this type. Most have three voices plus some non-canonic accompanying voices, but the final Agnus Dei of the mass \emph{Sine nomine} has a four-voice canon, each voice entering on a different pitch class, plus three accompanying voices for a 7-part texture in all. The canonic melody of this remarkable piece is constructed from a pentachord (D-E-F-G-A) and is imitated on the other three pentachords in the diatonic scale that have no melodic tritone (A-B-C-D-E, C-D-E-F-G, G-A-B-C-D). The time and pitch relation between the last two voices is the same as that between the first two. In his other asymmetric canons, the imitations are at perfect intervals (fourths, fifths, and octaves), as is typical of Renaissance polyphony more generally.

Most of these canons are accompanied by free voices, which ease the compositional process by filling in the texture during the frequent rests in the canonic parts; also, the free bass\footnote{Accompanied canons where the bass is one of the canonic parts are rare; Palestrina writes them only in the case of $2$-voice canon, and we will not consider them here.} corrects any fourths between the canonic parts and includes strong leaps of fourths and fifths at cadences, which could be awkward to imitate through each voice of the texture. However, the Pleni from the mass \emph{Ad fugam,} perhaps Palestrina's most impressive canon, is for three parts unaccompanied (see Figure \ref{fig:ad_fugam_pleni}). Except possibly for the final cadence, which introduces the so-called ``consonant 4th'' idiom by an atypical octave leap \citep[see][p.~236]{jeppesen2012style}, its polyphonic style features the same meticulous dissonance treatment practiced in Palestrina's free imitative works. Its scheme $\{(0,0)\tB,(1,3),(3,7)\}$ is the exact inversion of the scheme $\{(0,0),(1,-3),(3,-7)\tB\}$ used in one accepted realization of the famous anonymous English puzzle canon ``Non nobis Domine,'' once attributed to William Byrd (see Figure \ref{fig:non_nobis}).
\begin{figure}
  \includegraphics[width=\linewidth]{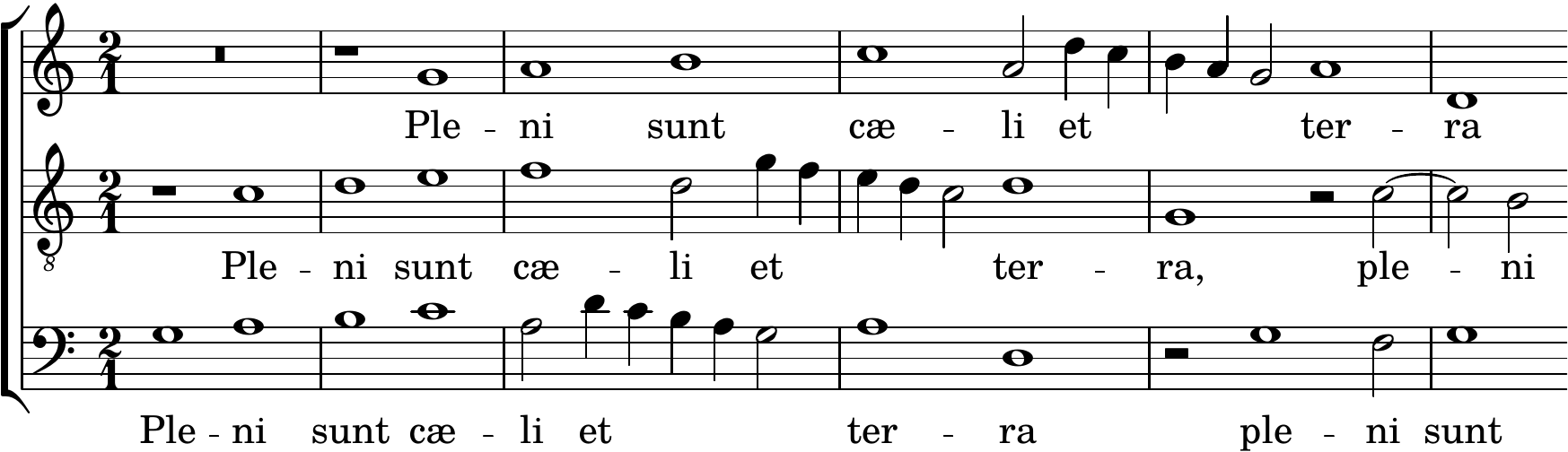}
  \caption{A canon from Palestrina's mass \emph{Ad fugam} with the asymmetric scheme $\{(0,0)\tB, (1,3), (3,7)\}$}
  \label{fig:ad_fugam_pleni}
\end{figure}
\begin{figure}
  \includegraphics[width=\linewidth]{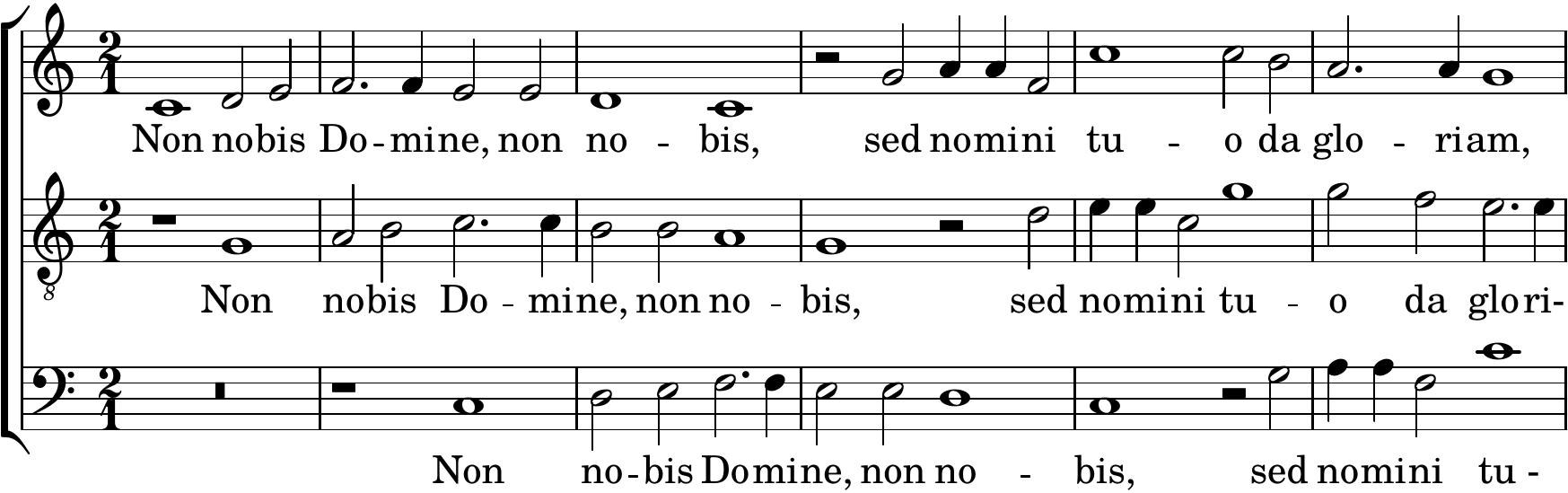}
  \caption{Realization of an anonymous English canon with the asymmetric scheme $\{(0,0),(1,-3),(3,-7)\tB\}$}
  \label{fig:non_nobis}
\end{figure}

\section{Defining the flexibility of a canonic scheme}
\label{sec:theor}
\subsection{A simplified model of the style}
To understand how Palestrina and his contemporaries chose the time and pitch displacements for their canons, we follow \citet{GauldinComposition} and other authors in making a first-species reduction of the canon; that is, we assume a melody of whole notes only. We simplify further by ignoring all octave differences and chromatic alterations, so that our tones belong to the finite field $\FF_7 = \ZZ/7\ZZ$, which consists of seven elements $\{0,1,2,3,4,5,6\} = \{B,C,D,E,F,G,A\}$. We thus depart from recent authors \citep{Geenen, Arias}, who use the ring $\ZZ/12\ZZ$, representing the chromatic scale, to analyze Renaissance-style counterpoint. The $7$-note diatonic scale is a better fit with the transposition relations that relate voices in a canon in this style (as is especially evident in Figure \ref{fig:sacerdotes_pleni}) and also simplifies our computations.

\begin{defn}
  A \emph{canonic scheme} $\S$ with $r$ \emph{voices} is a set of ordered pairs $\{(t_i, p_i)\}_{1 \leq i \leq r}$ with time displacements $t_i \in \ZZ$, each pitch displacement $p_i \in \FF_7$, with no two $t_i$ equal, and with one of the ordered pairs optionally bearing the annotation $\tB$ for ``bass.'' A \emph{melody} is a sequence $(x_i)_{i=1}^n = [x_1, x_2, \ldots, x_n]$ of \emph{notes} $x_i \in \FF_7$. Given a melody $(x_i)$, its \emph{realization} under the scheme $\S$ is the matrix of tones $[y_{ti}]$, where $y_{ti}$, the note sung by the $i$th voice at time $t$, is given by
  \begin{equation}
    y_{ti} = x_{t - t_i} + p_i.
  \end{equation}
\end{defn}

We then model the rules of dissonance avoidance in the style as follows:
\begin{defn}
  A melody is a \emph{valid canon} for a canonic scheme $\S$ if its realization obeys the following two rules for all voices $i$ and $j$, at all times $t \in \ZZ$ such that $y_{ti}, y_{tj}$ both exist (that is, such that $1 \leq t - t_i \leq n$ and $1 \leq t - t_j \leq n$):
  \begin{enumerate}[$($a$)$]
    \item\label{it:2nd} Harmonic seconds and sevenths are not allowed:
    \[
    y_{ti} - y_{tj} \notin \{1, 6\}.
    \]
    \item\label{it:4th} Fourths above the bass are not allowed:
    \[
    (t_j, p_j)\tB \implies y_{ti} - y_{tj} \neq 3.
    \]
  \end{enumerate}
  Let $\V_n(\S)$ denote the set of valid $n$-note canons for the scheme $\S$, and let $V_n(\S) = \size{\V_n(\S)}$, be the number of valid $n$-note canons. We abbreviate $V_n(\S)$ to $V_n$ when the choice of scheme $\S$ is clear.
\end{defn}

\begin{examp} \label{ex:2voice}
	Let $\S = \{(0,0), (t, p)\tB\}$ be a two-voice canonic scheme, $t > 0$. We have
	\[
	V_n(\S) = \begin{cases}
		7^n, & n \leq t \\
		7^t 4^{n-t}, & n > t
	\end{cases}
	\]
	because the first $t$ notes are completely free, and thereafter we may write the third, fifth, sixth, or octave above the already determined bass note.
\end{examp}

Note the vast simplification of the resources of Renaissance composition in this model. There are no rules forbidding tritones (harmonic or melodic), which are often avoided by \emph{musica ficta} chromatic alterations, which were typically supplied by the performer and need not match among the different voices of a canon. There is no provision for the use of rests; at any rate, a canonic scheme which forced the use of frequent rests could not be judged a good one. There is no imposition of the rule against parallel octaves and fifths, which can often be evaded by contrary motion or ornamenting the melodic line.
There is no provision for suspensions, which can and do complicate matters by allowing dissonance to occur on a strong beat. For example, Figure \ref{fig:versatile}
\begin{figure}[htbp]
	\centering
	\includegraphics[scale=1.0]{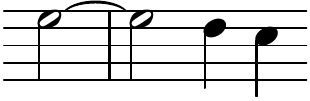}
	\caption{A versatile motif for canons}
	\label{fig:versatile}
\end{figure}
is a useful motif in asymmetric canons because any of its three notes may dissonate according to the laws of the style: the first as a suspension, the second as an accented passing note, and the third as a passing note, neighbor note, or cambiata.

\subsection{An example with Fibonacci behavior}
\label{sec:fib}

Consider the scheme $\S_\Fib = \{(0,0),(1,-8)\tB,(3,0)\}$, marked \ref{ti:fib} in Table \ref{tab:flex}. In words, this is a canon for three voices, with one upper voice entering first, the bass entering after one beat a ninth lower, and the remaining upper voice entering after two more beats at the same pitch as the first voice. A little experimentation will show that the valid canons for this scheme are constrained as follows:
\begin{itemize}
  \item At each note, the melody must move either up a step ($\up 1$), down a step ($\down 1$), up a fourth, that is, three steps ($\up 3$), or down a fourth ($\down 3$).
  \item Moreover, these motions cannot occur in arbitrary order, but an $\up 3$ can be followed by an $\up 1$ or another $\up 3$, while an $\up 1$ \emph{must} be followed by an $\up 3$, and so on, following the arrows in Figure \ref{fig:fib_graph}.
\end{itemize}
\begin{figure}[htbp]
  \centering
  \begin{tikzpicture}[
    scale = 1,
    every node/.style = {draw, circle, minimum size = 10mm}
    ]
    \node[] (A) {$\up3$};
    \node[right=15mm of A] (B) {$\up1$};
    \node[right=20mm of B] (C) {$\down3$};
    \node[right=15mm of C] (D) {$\down1$};
    \draw[->] (A) to[bend left] (B);
    \draw[->] (B) to[bend left] (A);
    \draw[->] (C) to[bend left] (D);
    \draw[->] (D) to[bend left] (C);
    \draw[->] (A) to[out=-150,in=150,loop] ();
    \draw[->] (C) to[out=-150,in=150,loop] ();
  \end{tikzpicture}
  \caption{A graph of the valid melodic motions for the canonic scheme $\S_\Fib = \{(0,0),(1,-8)\tB,(3,0)\}$ of Section \ref{sec:fib}. (Intervals are zero-based, so that, for instance, $\up3$ means that the melody rises by a fourth.)}
  \label{fig:fib_graph}
\end{figure}
\begin{figure}
  \centering
  \includegraphics[width=0.48\linewidth]{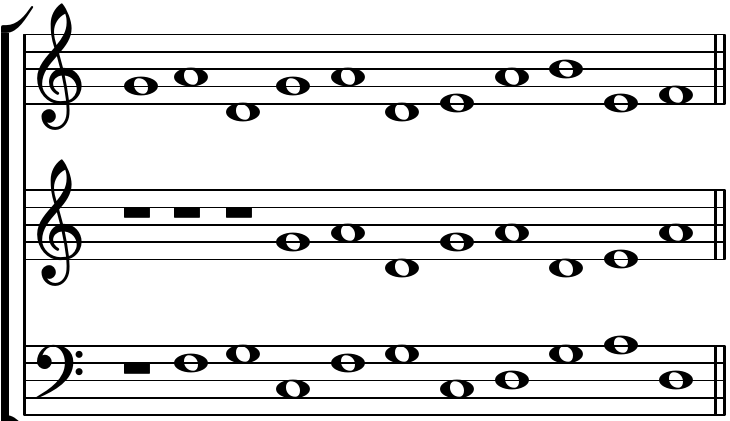}
  \hfill
  \includegraphics[width=0.48\linewidth]{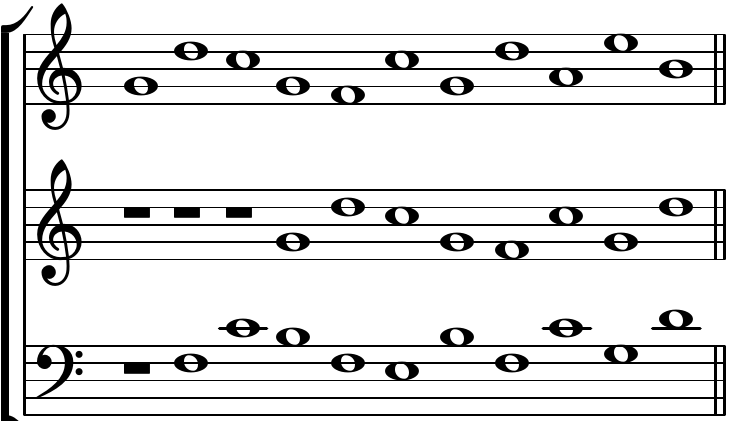}
  \caption{Two whole-note canons with the scheme $\S_\Fib$, which are valid according to the simplified model.}
  \label{fig:fib_exs}
\end{figure}
Two examples of valid canons for this scheme are shown in Figure \ref{fig:fib_exs}. Since the graph in Figure \ref{fig:fib_graph} is disconnected, the canons are of two types, one using only the moves $\up3$ and $\up1$, the other using only the moves $\down3$ and $\down1$.

To compute $V_n(\S_\Fib)$, we multiply the number of paths with $n-1$ nodes in the graph by $7$ (for the possible starting notes) to obtain the data shown in Table \ref{tab:fib}. Observe that, excluding the first term $V_1$, the results are $14$ times the Fibonacci numbers $1,1,2,3,5,\ldots,$ and are governed by the recurrence relation
\begin{table}
  \centering
\begin{tabular}{c|c|c}
  $n$ & $V_n$ & $V_n/14$ \\ \hline
  $1$ & $7$ & \\
  $2$ & $28$ & $2$ \\
  $3$ & $42$ & $3$ \\
  $4$ & $70$ & $5$ \\
  $5$ & $112$ & $8$
\end{tabular}
\caption{The count of valid canons for $\S_\Fib$. The Fibonacci sequence appears.}
\label{tab:fib}
\end{table}
\begin{equation} \label{eq:fib_rec}
  V_n = V_{n-1} + V_{n-2}.
\end{equation}
Appealing to the well-known formula for Fibonacci numbers, we have
\begin{equation} \label{eq:fib_exp}
  V_n = \frac{14}{\sqrt{5}} \big(\phi^{n-1} - \bar\phi^{n-1}\big)
\end{equation}
where $\phi = (1 + \sqrt{5})/2 = 1.618\dots$ is the golden ratio and $\bar\phi = (1 - \sqrt{5})/2 = -0.618\dots$ is its conjugate. Since the powers of $\bar\phi$ tend quickly to $0$, we can write the asymptotic equality
\begin{equation} \label{eq:fib_asymp}
  V_n \sim c \cdot \phi^n
\end{equation}
where $c = 14/(\phi\sqrt{5})$ is a constant. Informally, \eqref{eq:fib_asymp} states that under the scheme $\S_\Fib$, there are ``$\phi$ possibilities for each note''; in practice, some notes are forced, while others have two possibilities. In the next subsection, we make sense of the ``number of possibilities for each note'' for a general canonic scheme.
 
\subsection{The flexibility}
We will prove that, for every scheme $\S$, one can tell a similar story. There is a directed graph $\G_\S$, analogous to Figure \ref{fig:fib_graph}, with a finite (possibly huge) number of nodes, such that a valid canon with the scheme $\S$ corresponds to a walk along the arrows of $\G_\S$; and the number $V_n(\S)$ obeys a linear recurrence, like \eqref{eq:fib_rec}, with integer coefficients; and there is an explicit formula, like \eqref{eq:fib_exp}, which expresses $V_n(\S)$ as a finite combination of exponentials. We focus on the leading exponential, which provides a way to formalize the ``number of possibilities for each note'' under a given scheme.

\begin{defn}
  Let $\S$ be a canonic scheme. Define the \emph{flexibility} $\lambda(\S)$ to be a parameter measuring how fast $V_n(\S)$ grows as $n$ increases:
  \[
  \lambda(\S) = \lim_{n \to \infty} \sqrt[n]{V_n(\S)},
  \]
  assuming that the limit exists; we will later prove that it always does (see Theorem \ref{thm:alg_int}).
\end{defn}

Some examples of this definition follow:
\begin{itemize}
  \item The scheme $\S = \{(0,0)\}$ prescribes a ``canon'' with only one voice. As the melody can proceed unrestricted among the seven scale tones, we have $V_n(\S) = 7^n$, so the flexibility is $\lambda(\S) = 7$. Observe that the flexibility of any scheme $\S$ must lie in the range
  \[
    0 \leq \lambda(\S) \leq 7.
  \]
  \item As computed in Example \ref{ex:2voice}, a two-voice scheme $\S$ at $t$ beats with a marked bass voice has $V_n(\S) = 7^t 4^{n-t}$, which grows like $4^n$ for large $n$, so its flexibility is $\lambda(\S) = 4$. If there is no marked bass voice, an analogous computation shows that $\lambda(\S) = 5$. As more voices are added, the flexibility evidently cannot increase, so our $\lambda$-values for canons of two or more voices will lie in the range
  \[
  0 \leq \lambda(\S) \leq 5.
  \]
  \item The three-voice scheme $\S_\Fib$ studied in Section \ref{sec:fib} has flexibility $\phi = 1.618\ldots.$
\end{itemize}

\section{Properties of the flexibility value}
\label{sec:properties}

The payoff for such a stark simplification of the laws of the style is that the flexibility of a canonic scheme is invariant under a large set of transformations, vastly reducing the possibilities we need to consider:

\begin{restatable}{prop}{propeqv}\label{prop:eqv}
  Let $\S$ be a canonic scheme, and let $\S'$ be a new scheme derived from it by one of the following transformations:
  \begin{enumerate}[$($a$)$]
    \item\label{it:ttrans} Time translation $t_i \mapsto t_i + a$, for a constant $a \in \ZZ$
    \item\label{it:ptrans} Pitch transposition $p_i \mapsto p_i + b$, for a constant $b \in \FF_7$
    \item\label{it:shear} ``Shearing'' $p_i \mapsto p_i + ct_i$, for a constant $c \in \FF_7$
    \item\label{it:inv} Inversion: keeping the time intervals $t_i' = t_i$, set
    \[
    p_i' = \begin{cases*}
      5 - p_i & if $(t_i, p_i)$ is a marked bass voice \\
      -p_i & otherwise.
    \end{cases*}
    \]
    \item\label{it:tdil} Time dilation (augmentation, diminution, or retrograde) $t_i \mapsto c t_i$, for a constant $c \in \QQ^\cross$ for which all $ct_i \in \ZZ$.
  \end{enumerate}
  Then $\S'$ and $\S$ have the same flexibility value:
  \[
  \lambda(\S') = \lambda(\S).
  \]
\end{restatable}

Of these operations, only ``shearing'' is a novelty in musical analysis. If the constant $c$ is $1$, then the shearing operation moves the first note up one step, the second note up two steps, and so on (see Figure \ref{fig:shear_mel}).
\begin{figure}
  \centering
  \includegraphics[width=\linewidth]{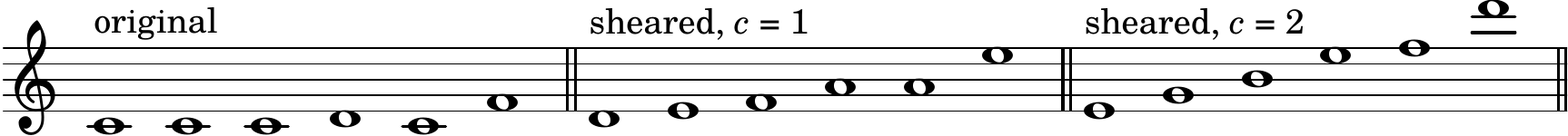}
  \caption{A ``shearing'' operation on melodies. A melody (left) is restated with the $i$th note moved up $i$ steps (center) or $2i$ steps (right).}
  \label{fig:shear_mel}
\end{figure}
The vertical intervals $y_{ti} - y_{tj}$ are thereby preserved, although the new canon may have quite a different musical flavor from the original (see Figure \ref{fig:shear_canon}).
\begin{figure}
  \centering
  \includegraphics[width=0.48\linewidth]{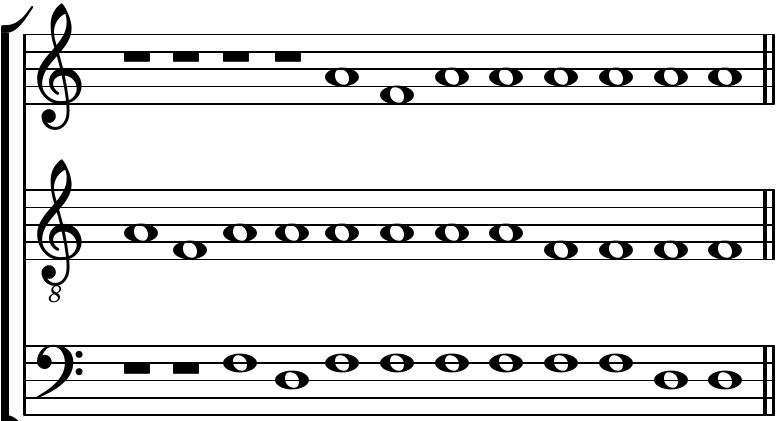}
  \hfill
  \includegraphics[width=0.48\linewidth]{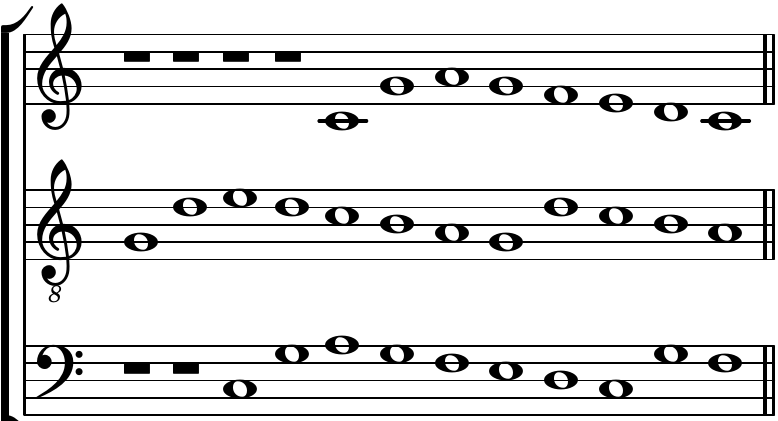}
  \caption{Canons with the schemes $\{(0,0), (2,-2)\tB, (4,7)\}$ (left) and $\{(0,0), (2,-4)\tB, (4,3)\}$ (right), which both invert at the 12th. The two canons are related by a shear transformation with $c = -1$ that transposes the chord on the $i$th beat by $i$ steps downward, with some octave adjustments. The shear operation transforms a melody of just two notes (left) into the familiar folk melody of ``Twinkle, Twinkle, Little Star'' (right).}
  \label{fig:shear_canon}
\end{figure}

A consequence of the time dilation invariance, item \ref{it:tdil}, is that the time unit used to measure the displacements $t_i$ is immaterial. In triple meter, for example, we can measure $t_i$ in beats or in bars without affecting the computed flexibility value. (This is not true under more subtle notions of validity, for instance if parallel perfect intervals are disallowed.)

The proof of this proposition is somewhat lengthy and will be found in Section \ref{sec:propeqv}.

\section{Flexibility data for three-voice schemes}
\label{sec:data}
We call two canonic schemes \emph{equivalent} if they can be transformed to one another by some combination of the transformations in Proposition \ref{prop:eqv}. In particular, we know that equivalent canonic schemes have the same flexibility.

In Figure \ref{fig:flex}, we graph the flexibility values of all three-voice canons $\{(t_1,p_1), (t_2,p_2), (t_3,p_3)\}$ with time displacements $t_i$ up to $8$, which conveniently is enough to include all three-voice canons in Table \ref{tab:Pal_canons} after suitably rescaling the time unit. The same flexibility values are shown numerically in Table \ref{tab:flex} in Appendix \ref{app:tab}. Note that we can make great reductions in the number of cases using Proposition \ref{prop:eqv}. First, we can use time and pitch translation to assume that $t_1 = p_1 = 0$ (we always assume $t_1 < t_2 < t_3$). Then, we can use time dilation to assume that $t_2$ and $t_3$ are coprime positive integers. We can also use time dilation by $-1$ to assume that $t_3 \geq 2 t_2$ (that is, the last two voices are at least as far apart in time as the first two). Finally, we can use shearing to set $p_2 = 0$, unless $t_2$ is a multiple of $7$, in which case we can set $p_3 = 0$ (this last case does not occur in the range of values under consideration). Thus we have only three free variables $t_2, t_3, p_3$, together with the option to designate one of the three voices as bass. In Figure \ref{fig:flex} and Table \ref{tab:flex}, we have sorted the possible schemes first by $t_3$, then by $t_2$ (which runs over the positive integers up to $t_2/2$ that are coprime to $t_3$). More information about how these values were computed will be found in Section \ref{sec:computing}.

In Table \ref{tab:rep_flex}, we list the flexibility values for the Palestrina canons in Table $1$, plus some canons by other Renaissance composers for comparison. 

\begin{figure}
  \begin{tikzpicture}[>=latex, shorten >= 2pt]
    \tikzset{
      mynode/.style={rectangle,rounded corners,draw=black, top color=white, bottom color=yellow!50,very thick, inner sep=1em, minimum size=3em, text centered},
      myarrow/.style={->, >=latex', shorten >=20pt},
      mylabel/.style={text width=7em, text centered} 
    }  
    \begin{axis}[
      width=\textwidth,
      xlabel={Time displacements $(t_2,t_3)$},
      ylabel={Flexibility $\lambda(\mathcal{S})$},
      xtick=\empty,
      extra x ticks={%
        0, 2, 4, 6,7, 9, 11,12,13, 15,16
      },
      extra x tick style={
        tick label style={
          rotate=90,anchor=east}},
      extra x tick labels={%
        ${(1,2)}$,
        ${(1,3)}$,
        ${(1,4)}$,
        ${(1,5)}$,${(2,5)}$,
        ${(1,6)}$,
        ${(1,7)}$,${(2,7)}$,${(3,7)}$,
        ${(1,8)}$,${(3,8)}$%
      },
      legend entries={%
        no bass,%
        {no bass, found in repertoire},
        with bass,%
        {with bass, found in repertoire},
      }
      ]
      \addplot[
      scatter,
      only marks,
      point meta=explicit symbolic,
      scatter/classes={
        n={mark=diamond,mark size=3,blue},
        nhl={mark=diamond*,mark size=3,blue},
        b={mark=triangle,mark size=3,red},
        bhl={mark=triangle*,mark size=3,red}
      },
      ] table [meta=label,x=x,y=flex] {flex.dat};
      %
      \node[] (stacked1) at (axis cs: 0, 3.935) {};
      \node[] (stacked2) at (axis cs: 0, 3.140) {};
      \node[above right=0in and 0.22in of stacked1] (stacked-label) {stacked canons};
      \draw[->] (stacked-label.west) -- (stacked1);
      \draw[->] (stacked-label.west) -- (stacked2);
      
      \node[] (Forestier) at (axis cs: 0, 3.562) {};
      \node[above right=0.1in and 0.2in of Forestier] (Forestier-label) {%
        $\vee$-canons at the 10th
      };
      \draw[->] (Forestier-label.west) -- (Forestier);
      
      \node[] (impos) at (axis cs: 0, 1) {};
      \node[right=of impos, node distance=1in] (impos-label) {inflexible canon (see Figure \ref{fig:impos})};
      \draw[->] (impos-label) -- (impos);
      
      \node[] (AFP) at (axis cs: 2, 2.992) {};
      \node[above right=0.58in and 0.2in of AFP] (AFP-label) {\emph{Ad fugam,} Pleni};
      \draw[->] (AFP-label.south west) -- (AFP.east);
      
      \node[] (Non) at (axis cs: 2, 2.781) {};
      \node[below right=0.7in and 0.2in of Non] (Non-label) {``Non nobis Domine''};
      \draw[->] (Non-label.north west) -- (Non.east);
      
      \node[] (New) at (axis cs: 13, 2.814) {};
      \node[below right=0.7in and 0.2in of New,anchor=north] (New-label) {New canon (see Section \ref{sec:new_sc})};
      \draw[->] (New-label.north) -- (New.east);
      
      \node[] (Fib) at (axis cs: 2, 1.618) {};
      \node[below right=0.05in and 0.5in of Fib,anchor=west] (Fib-label) {New canon (see Section \ref{sec:new_fib})};
      \draw[->] (Fib-label.west) -- (Fib.east);
      
      \node[] (inv) at (axis cs: 0, 3) {};
      \node[below right=1.3in and 0.3in of inv,anchor=north west] (inv-label) {$\vee$-canons at the 12th};
      \draw[->,black,double] (inv-label.north west) -- (inv);
      
      \node[] (Herc) at (axis cs: 0, 2) {};
      \node[below right=0.8in and 0.3in of Herc,anchor=north west] (Herc-label) {$\wedge$-canons at the 12th};
      \draw[->] (Herc-label.north west) -- (Herc);
      
    \end{axis}
  \end{tikzpicture}
  \caption{Flexibility values for various three-voice canonic schemes. Along the $x$-axis are shown the possible time displacements $(t_2, t_3)$ that satisfy $\gcd(t_2, t_3) = 1$ and $2t_2 \leq t_3$, which we can assume. For each $(t_2, t_3)$, we show a column of data points in which the pitch displacement $p_2$ and the bass annotation $\tB$ vary.}
  \label{fig:flex}
\end{figure}

\begin{table} \label{tab:rep_flex}
  \begin{tabular}{rl|l}
    Composer & Work & Flexibility \\ \hline
    many & all 3-part accompanied stacked canons & 3.935 \\
    Forestier & \emph{L'homme arm\'e,} Agnus Dei I & 3.562 \\
    Forestier & \emph{L'homme arm\'e,} Agnus Dei II & 3.542 \\
    Palestrina & ``Virgo prudentissima'' (I pars) & 3.363 \\
    Palestrina & \emph{Primi toni,} Agnus Dei II & 3.344 \\
    Palestrina & \emph{Ad fugam,} Agnus Dei II & 3.292 \\
    Palestrina & \emph{Sine nomine,} Agnus Dei II & 3.177 \\
    many & all 3-part unaccompanied stacked canons & 3.140 \\
    Palestrina & ``Accepit Jesus calicem'' & 3.118 \\
    Palestrina & $\vee$-canons inverting at the 12th & 3 \\
    Palestrina & \emph{Ad fugam}, Pleni & 2.992 \\
    O. & New canon (see Section \ref{sec:new_sc}) & 2.814 \\
    anon. & ``Non nobis Domine'' & 2.781 \\
    Palestrina & \emph{Repleatur,} Agnus Dei II & 2.754 \\
    La Rue & \emph{O Salutaris Hostia,} Kyrie II & 2.683 \\
    La Rue & \emph{O Salutaris Hostia,} Christe & 2.678 \\
    La Rue & \emph{O Salutaris Hostia,} Kyrie I & 2.420 \\
    Forestier & \emph{L'homme arm\'e,} Agnus Dei III & 2.420 \\
    Josquin & \emph{Hercules dux Ferrariae,} Agnus Dei II & 2 \\
    O. & New canon (see Section \ref{sec:new_fib}) & 1.618
  \end{tabular}
  \caption{The flexibility values in canons by Palestrina and other composers}
\end{table}

Several trends can be seen in Figure \ref{fig:flex}. Both the highest and the lowest flexibility values occur near the left side of the graph, where $t_3$ is small, that is, the time displacements are in a simple ratio. The leftmost column, marked $(1,2)$, comprises the symmetric canons. In the top left data points (entries \ref{ti:stacked}, \ref{ti:inv}, and \ref{ti:stacked-u} of Table \ref{tab:flex}) we recover that stacked canons, as well as inverting $\vee$-canons at the octave, 10th, and 12th with the bass entering second, are very flexible, having flexibility values of $3$ or greater. By contrast, the canonic scheme $\{(0,0)\tB, (1,11), (2,7)\}$, appearing in the bottom left corner of Figure \ref{fig:flex}, has a flexibility value of exactly $1$ (the minimum possible nonzero flexibility value). It is an inverting $\wedge$-canon at the 9th. For any length $n \geq 2$, it admits only \emph{two} valid canons up to transposition, both of minimal musical interest, as shown in Figure \ref{fig:impos}. Unsurprisingly, there is \emph{no} canon in the Renaissance repertoire utilizing this scheme or an equivalent one.
\begin{figure}
  \centering
  \hfill
  \includegraphics[width=0.4\linewidth]{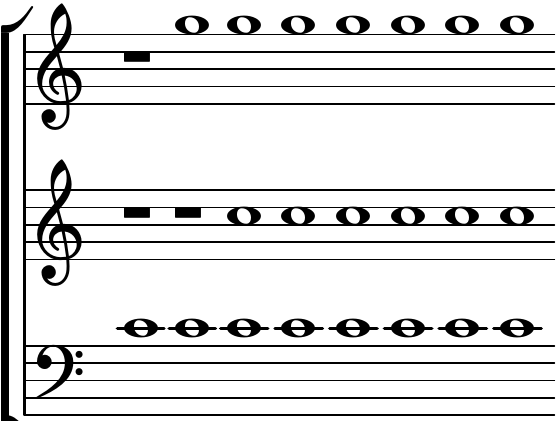}
  \hfill
  \includegraphics[width=0.4\linewidth]{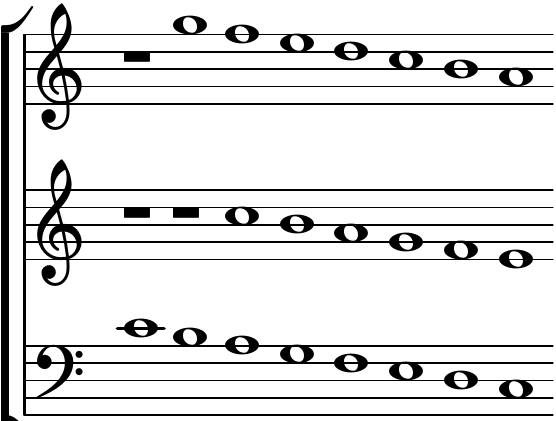}
  \hfill{}
  \caption{The only two consonant canons with the scheme $\S = \{(0,0)\tB, (1,4), (2,7)\}$, which has flexibility $1$.}
  \label{fig:impos}
\end{figure}

We also obtain a discrepancy between the unaccompanied inverting canons of $\vee$ and $\wedge$ type employing inversion at the same interval. For example, for inversion at the 12th, the $\vee$ type, used by Palestrina (see Figure \ref{fig:GiaFuBenedictus}), has flexibility $3$ while the $\wedge$ type has a flexibility value of only $2$. The latter is the scheme of the Agnus Dei II\footnote{Unlike Palestrina, the composers of the Josquin generation often wrote three separate Agnus Dei movements in his masses, the middle Agnus Dei typically using reduced forces to prepare for a grand Agnus Dei III.} of Josquin's mass \emph{Hercules dux Ferrariae} \citep*{Josquin}, which includes quite a few concessions in the dissonance treatment, such as upward-resolving suspensions, which look backward to Ockeghem rather than forward to Palestrina. This is the lowest flexibility value we have found in the repertoire analyzed.  The \emph{secunda pars} of Palestrina's motet ``Virgo prudentissima'' uses a similar scheme, a $\wedge$-canon inverting at the 12th, as listed in Table \ref{tab:Pal_canons}; but it is an accompanied canon, and in this case the $\vee$ and $\wedge$ types have equal flexibility values of $3$.

In short, the composers of Josquin's generation make use of a wide range of flexibility options from $2$ up to $3.935$ in canons for three-voice canons; Palestrina excludes the lower portion of this, his lowest $\lambda$-value of $2.754$ being found in the $6$-part Agnus Dei of the mass \emph{Repleatur os meum,} where the three canonic parts have lengthy rests.

Incidentally, the appearance of exact integers among the flexibility values correlates nicely with the existence of simple rules for composing canons. For instance, the value $\lambda = 3$ for $\vee$-canons at the 12th appears because one can write the octave, third, or fifth above the bass, following the long-known rule for invertible counterpoint at the 12th. As the displacements $t_i$ increase, the $\lambda$-values become algebraic integers of higher and higher degrees, and it is unreasonable to expect that a simple rule will enumerate all the valid canons.

One tendency of the repertoire \emph{not} captured by the model, but familiar to theorists in the Renaissance and after, is the preference for canons at perfect intervals; note that in Table \ref{tab:Pal_canons}, most pieces use pitch displacements by $p \equiv 0, 3, 4$ mod $7$, that is, imitation at the fourth, fifth, and octave, and that $p \equiv 3$ and $p \equiv 4$ do not normally occur in the same piece. In our model, owing to the shearing symmetry, any interval is as good as any other. The preference for fourth and fifth imitations cannot, therefore, be motivated by flexibility of the canon as we have defined it. It must be traced rather by a desire to imitate melodic motives more exactly, in that a melody based on a Guidonian hexachord (such as C-D-E-F-G-A) can be imitated at the fifth or fourth on another hexachord (such as G-A-B-C-D-E) without altering the placement of whole and half steps, and avoiding the melodic tritone to boot.

Another habit of Palestrina's, revealed by careful inspection of Table \ref{tab:Pal_canons}, is to bring in two voices in tight stretto, typically one semibreve apart at the interval of a fourth, while bringing in the third voice at a longer delay either before or after. This is equivalent to setting $t_2 = 1$ (or $t_2 = t_3 - 1$, which is equivalent under retrograde). In view of Figure \ref{fig:flex}, this preference has no perceptible effect on the computed flexibility value. However, it may help in producing canons with more stepwise movement and fuller harmonies, inasmuch as steps in the melody produce rich thirds and fifths against the tightly imitating line.

As the time displacements $t_i$ increase while remaining relatively prime, the flexibility values appear to even out and tend toward a limit, which is apparently around $3.3$ when there is no canonic bass voice and around $2.8$ when there is. Palestrina mainly draws from the former region, but the Pleni from his mass \emph{Ad fugam,} whose opening was shown earlier in Figure \ref{fig:ad_fugam_pleni}, successfully uses the second most flexible scheme ($\lambda(\S) = 2.992$) among all unaccompanied asymmetric three-voice canons within the limits of our calculations! It is absurd to suggest that Palestrina chose this scheme by an exhaustive mathematical analysis of the possibilities as we have done here. Most likely, his compositional skill and vast experience with imitative polyphony suggested a theme and a workable canonic scheme for it. By contrast, the anonymous riddle canon ``Non nobis Domine'' uses the inversion of this scheme, which has the significantly lower flex value of $2.781$. So it is perhaps inevitable that this piece is short and essentially constructed from a single motive, as Gauldin points out \citep*{GauldinComposition}.

The only more flexible scheme for an unaccompanied, asymmetric three-voice canon is given by $\{(0,0)$, $(1,0)\tB$, $(3,4)\}$ (or any of its numerous equivalent forms). It can be obtained from a 4-voice inverting canon at the 12th with the scheme $\{(0,0), (1,0)\tB, (2,4), (3,4)\}$ (a scheme frequently used by Josquin) by omitting the third voice. Is there a canon somewhere in the Renaissance repertoire using this scheme?

\section{Computing the flexibility value}
\label{sec:computing}

In the following theorem, we show that the limit defining the flexibility is well-defined, and we provide an effective method to compute it. Moreover, although it has lesser musical significance, one cannot help noticing that some of the flexibility values in Table \ref{tab:flex} have number-theoretic importance: these include the exact integers $1$, $2$, and $3$, and some quadratic surds, such as the golden ratio $\frac{1 + \sqrt{5}}{2} = 1.618\ldots$ (as noted in Section \ref{sec:fib}) and $1 + \sqrt{3} = 2.732\ldots.$ These are all algebraic integers (roots of integer polynomials with leading coefficient $1$).

\begin{thm} \label{thm:alg_int}
  For any canonic scheme $\S$ with at least two voices, the flexibility $\lambda(\S)$ is an algebraic integer, that is, a root of an integer polynomial with leading coefficient $1$, and the degree of this polynomial is at most $ \frac{1}{7} V_s(\S) \leq 7^{s - 1}$, where
  \[
  s = \max_i t_i - \min_i t_i
  \]
  is the total time displacement of the canon.
\end{thm}
\begin{proof}
  The technique of the proof uses the Perron--Frobenius theorem. A suitable form of the theorem is as follows:
  \begin{defn}[see \cite{MeyerMatrix}, \textsection8.3]
    A square matrix $A$ is called \emph{reducible} if, after permuting the rows and columns the same way, it becomes a block-triangular matrix
    \begin{equation} \label{eq:XYZ}
      \begin{bmatrix}
        X & Y \\
        0 & Z
      \end{bmatrix}
    \end{equation}
    with $X$ and $Z$ square, and \emph{irreducible} otherwise. For an $n\times n$ matrix $A$, define its \emph{graph} $\G_A$ to be the (directed, possibly non-simple) graph on $n$ nodes with an edge from node $i$ to node $j$ for each nonzero entry $a_{ij}$. Observe that $A$ is irreducible exactly when $\G_A$ is \emph{strongly connected,} that is, for any nodes $i$ and $j$, there exists a directed path from $i$ to $j$.
  \end{defn}
  \begin{prop}[Perron--Frobenius theorem; see \cite{MeyerMatrix}, \textsection8.3]
    Let $A$ be an irreducible square matrix with nonnegative entries. Then:
    \begin{enumerate}[$($a$)$]
      \item The largest positive eigenvalue $\lambda$ of $A$ is also the spectral radius of $A$, that is, has the largest absolute value of any eigenvalue of $A$.
      \item $\lambda$ has algebraic multiplicity $1$.
      \item The corresponding eigenvector $\mathbf{v}$ is nonnegative, and is the unique nonnegative eigenvector (up to scaling) of $A$.
    \end{enumerate}
  \end{prop}
  
  Given a canonic scheme $\S$, we construct a graph $\G_\S$ and a matrix $A_\S$ as follows. Let $\G_\S$ have as nodes the $N = \frac{1}{7}V_s(\S)$ valid canons for $\S$ of length $s$ starting with the note $x_1 = 0$. For each of the $\frac{1}{7} V_{s+1}(\S)$ valid canons $(0 = x_1, x_2, \ldots, x_{s+1})$, draw a directed edge from
  \[
  (x_1, x_2, \ldots, x_s) \quad \text{to} \quad (x_2 - x_2, x_3 - x_2, \ldots, x_{s+1} - x_2).
  \]
  Thus $\G_\S$ is a directed graph on $N$ nodes, possibly with loops. Let $A_\S$ be the adjacency matrix of $\G_\S$; that is, an $N \cross N$ matrix whose rows and columns number the nodes and whose entries $a_{ij}$ count the number of edges pointing from node $i$ to node $j$. It is evident that $A_\S$ has nonnegative integer entries and encodes the same information as the graph $\G_\S$.
  
  Note also that $V_{n + s}(\S)$ is the number of paths with $n$ nodes in $\G_\S$: we can encode a canon $(x_1, \ldots x_{n + s})$ by following the edges corresponding to its substrings of length $s + 1$. Here it is essential that the validity of a canon can be checked by looking only at substrings of length at most $s + 1$, which is true by our choice of $s$. But, as is easy to see, counting paths amounts to taking powers of the adjacency matrix, so
  \begin{equation}
    \frac{1}{7} V_{n + s}(\S) = \1_N^{\top} A_\S^n \1_N, \label{eq:An}
  \end{equation}
  where $\1_N$ is the all-ones vector of length $N$. Thus we see the relevance of the eigenvalues of $A_\S$, which is a nonnegative integer matrix.
  
  However, $A_\S$ need not be irreducible. So we break up $\G_\S$ into its \emph{strongly connected components} $\G_i$ with sizes $N_i$. The adjacency matrix $A_i$ of $\G_i$ satisfies the hypotheses of the Perron--Frobenius theorem and so has a largest eigenvalue $\lambda_i$. Using the fact that the characteristic polynomial of a block-triangular matrix
  \[
  \begin{bmatrix}
    X & Y \\
    0 & Z
  \end{bmatrix}
  \]
  is the product of the characteristic polynomials of $X$ and $Z$, we see that the eigenvalues of $A_\S$ comprise the union (as a multiset) of the eigenvalues of the $A_i$. In particular, $A_\S$ has a dominant eigenvalue $\lambda$ which is real, nonnegative, and at least as great in absolute value as all other eigenvalues. Moreover, since $A_\S$ has integer entries (indeed, entries $0$ or $1$ unless $s = 1$), its characteristic polynomial
  \[
  f(x) = x^N + c_{N-1} x^{N-1} + \cdots + c_1 x + c_0
  \]
   has integer coefficients and leading coefficient $c_N = 1$. Hence $\lambda$ is an algebraic integer of degree at most $N \leq 7^{s - 1}$.
  
  To finish, we prove that $\lambda$ is in fact the flexibility:
  \begin{restatable}{lem}{lemeig} \label{lem:eig}
    The flexibility $\lambda(\S)$ of a scheme $\S$ is the largest eigenvalue $\lambda$ of its associated adjacency matrix $A_\S$.
  \end{restatable}
  The proof of this lemma is rather technical and will be found in Section \ref{sec:lemeig}.
\end{proof}

\begin{rem}
  Moreover, the Galois conjugates of $\lambda$ are among the roots of $f$ (possibly not all of them, as $f$ may be reducible), so, since $\lambda$ is the dominant root, it is at least as large in magnitude as each of its Galois conjugates. This yields a restriction on what algebraic numbers can show up at the flexibility of a canonic scheme.
\end{rem} 

This interpretation of the flexibility value also yields an efficient way to compute it. We first enumerate all valid canons of length at most $s + 1$ to find the graph $\G_\S$, which we then decompose into strongly connected components $\G_i$. For each $i$, we apply $A_i$ repeatedly to an arbitrary nonnegative vector (we can choose $\1_{N_i}$) until the result stabilizes to the desired precision. This has the advantage that $\lambda(\S)$ can be computed precisely and quickly once $A_\S$ is found, and the complexity of this grows exponentially with $s$. We have been able to compute examples for $s$ up to $8$, which conveniently is enough to include all the $3$-voice schemes in Table \ref{tab:Pal_canons}. 

%

\section{Canons with hitherto unexplored schemes} \label{sec:new}

To explore the viability of composition using the vast variety of canonic schemes predicted to be feasible by our model, we present two new compositions on canonic schemes which, apparently, have not been used by any composer before.

\subsection{The Fibonacci scheme}\label{sec:new_fib}

The first piece (Figure \ref{fig:fib_comp}) uses the scheme $\S_\Fib$ studied in depth in Section \ref{sec:fib}. It has a flexibility value of $\lambda = 1.618\ldots,$ lower than what we have computed in any published piece. The theoretical model provides a useful guide to composition: when the melody proceeds in whole notes and other unsyncopated note values, the downbeats follow the motifs illustrated in Figure \ref{fig:fib_exs} which are derived from the graph in Figure \ref{fig:fib_graph}. Using rests or suspensions, it is possible to cross over from one side of Figure \ref{fig:fib_graph} to the other. The listener will note heavy repetition of certain motives, such as the one in Figure \ref{fig:versatile}; our model explains this as an inevitable consequence of the limited flexibility of the scheme.

\begin{figure}[htbp]
  \includegraphics[width=\linewidth]{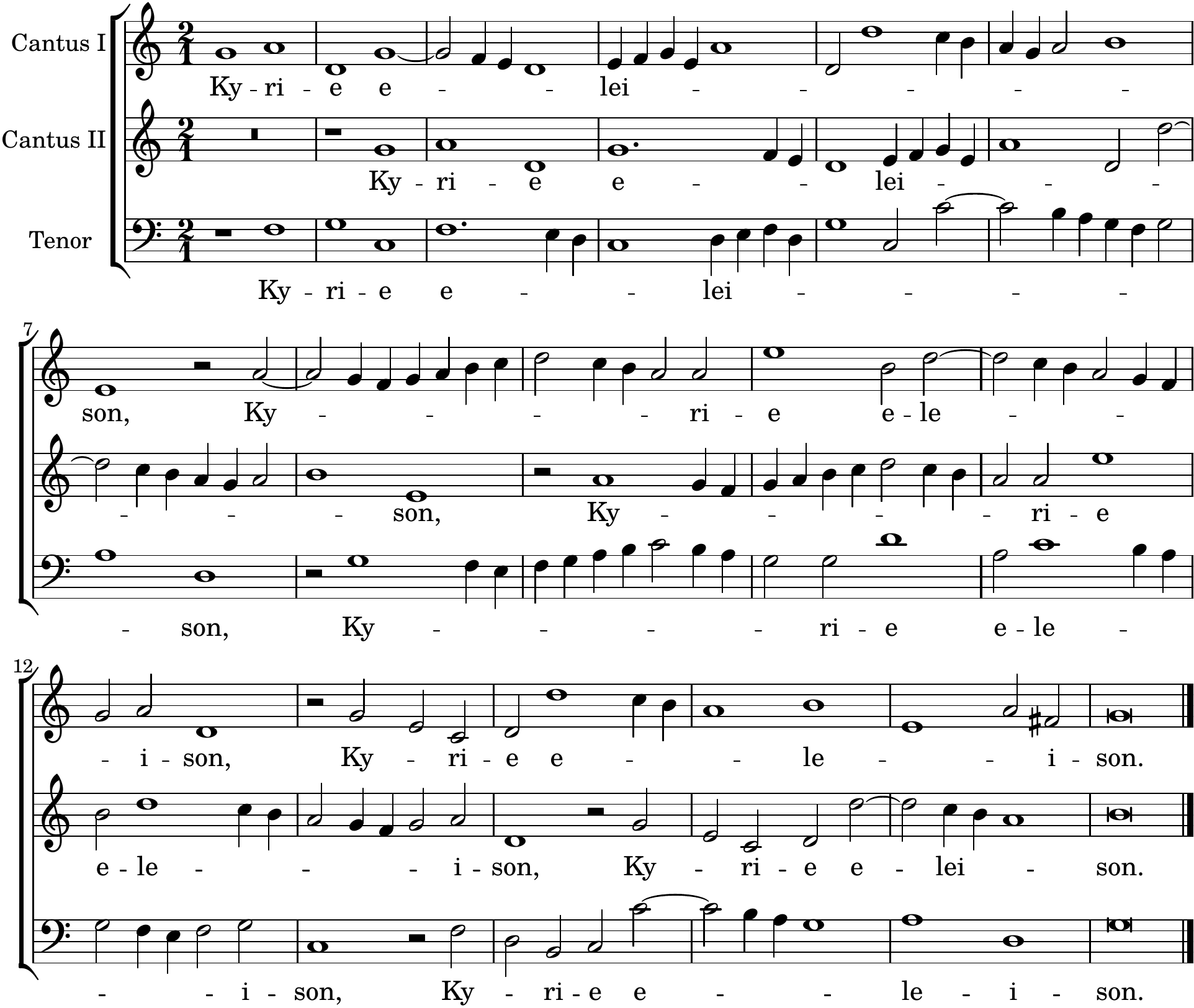}
  \caption{A new canon on the scheme $\S_\Fib$ of Section \ref{sec:fib}, with the limited flexibility value $\phi = 1.618\dots$}
  \label{fig:fib_comp}
\end{figure}

\subsection{A scheme from Palestrina's ``Sicut cervus''}\label{sec:new_sc}

Our second new composition (Figures \ref{fig:sc-begin}-\ref{fig:sc-end}) is a canon with the scheme
\[
\S_{\mathrm{SC}} = \{(0,0)\tB, (4,4), (7,7)\},
\]
It has a flexibility value of 2.814\dots (marked \ref{ti:new} in Table \ref{tab:flex}), which would fit in the mid-low range of Palestrina's output. It was not picked at random but rather is the imitative scheme for the first three voices of Palestrina's motet ``Sicut cervus,'' one of the most beloved Renaissance pieces today. This opens the way for another challenge frequently embraced by composers in the sixteenth century, namely incorporating preexisting material into a new piece. In the new piece, we weave in points of imitation on two themes derived from ``Sicut cervus,'' which is not a canon but a freely imitative motet (see Figures \ref{fig:SC1}--\ref{fig:SC2}). The new composition could thus serve as a movement for a parody mass (a frequent genre in the 16th century) based on this motet. In listening to the piece, observe that the scheme supports a wide variety of melodic and rhythmic effects within the style and that each of the preexisting themes admits multiple possible continuations.

\begin{figure}
  \centering
  \hfill
  \includegraphics[width=\linewidth]{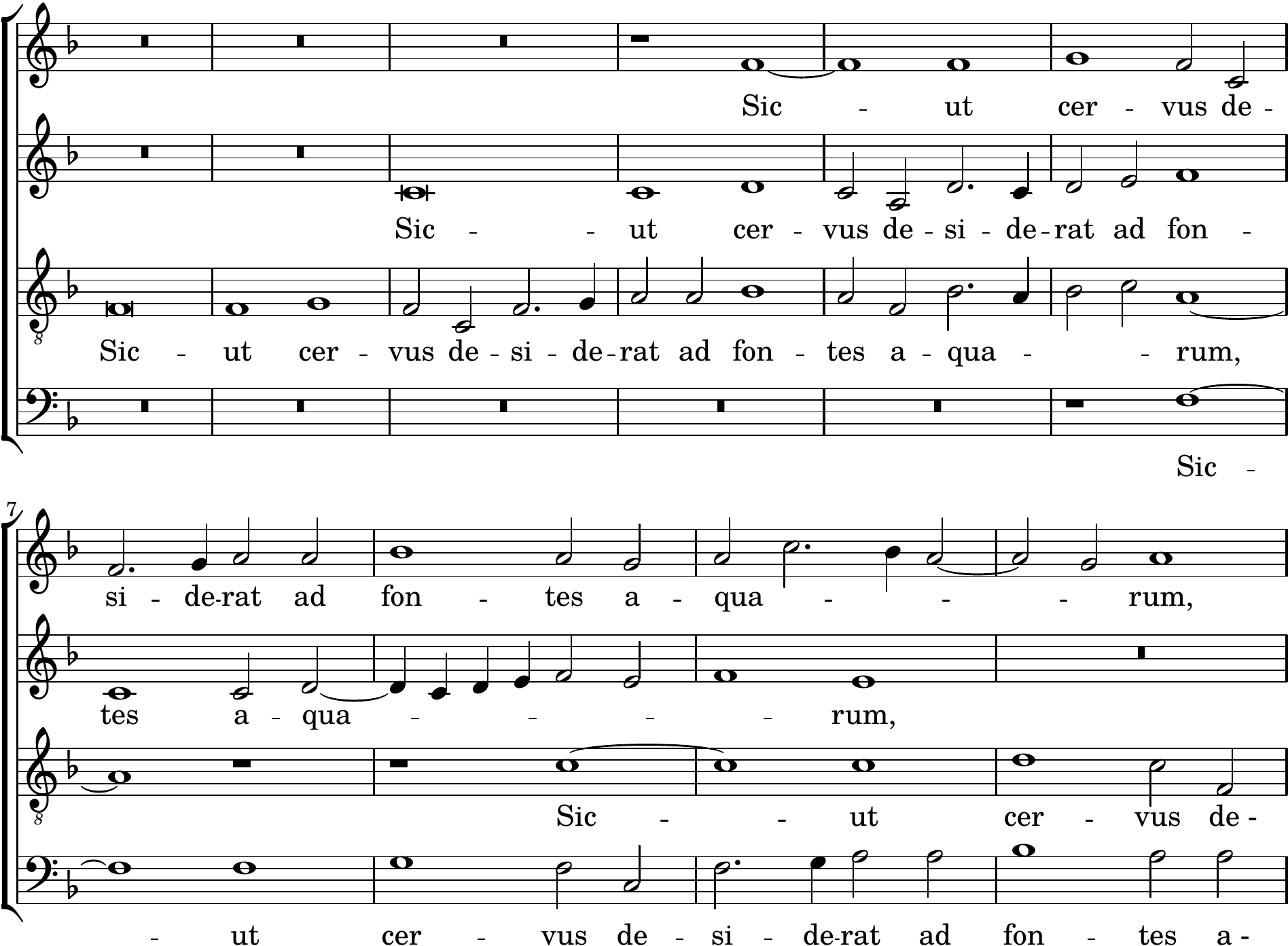}
  \caption{The opening of Palestrina's motet ``Sicut cervus.''}
  \label{fig:SC1}
\end{figure}

\begin{figure}
  \centering
  \hfill
  \includegraphics[width=\linewidth]{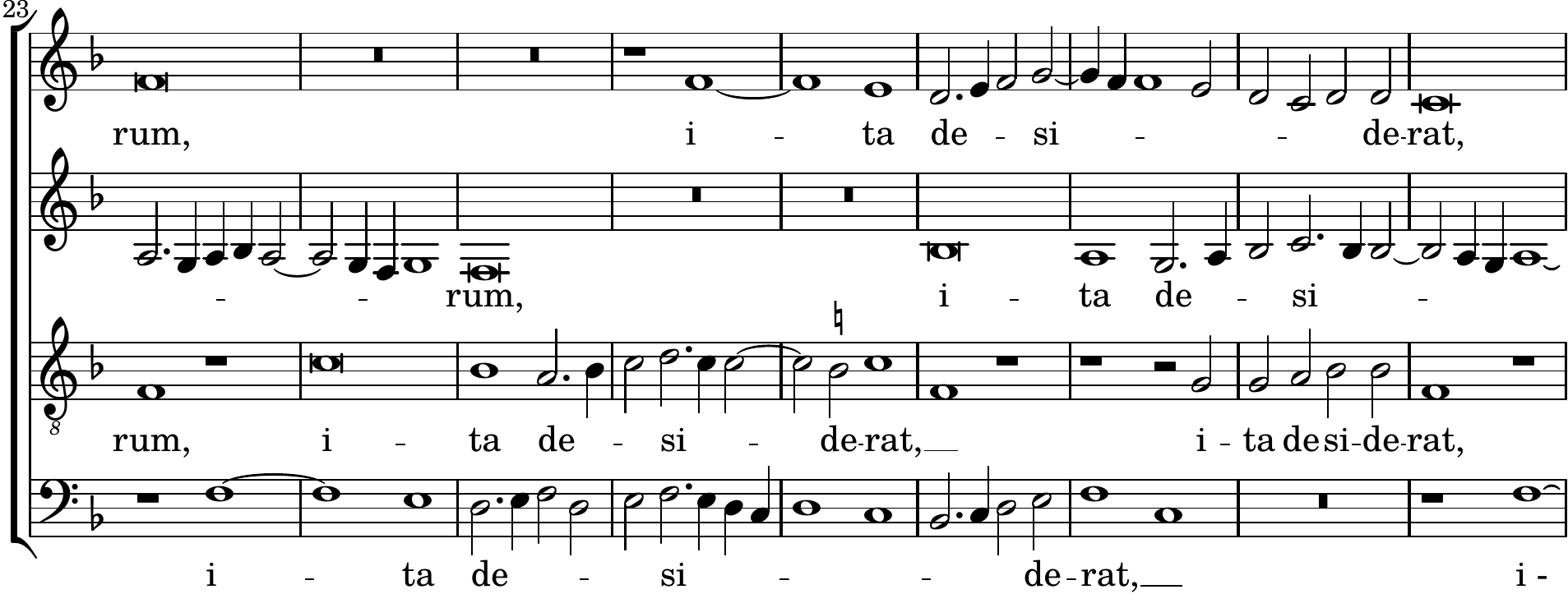}
  \caption{Another imitative passage in Palestrina's motet ``Sicut Cervus.''}
  \label{fig:SC2}
\end{figure}

As an aid in construction, we asked a computer to generate random canons on this scheme beginning with various note sequences, in other words, to take a random walk on the graph $\G_\S$. Figure \ref{fig:SC_gen} shows one of the many results that come from running this program starting from the seed values $(x_i) = (1,1,1,2,1,1,3,4)$ (taken from the first eight downbeats of the tenor voice of Figure \ref{fig:SC1}). Some motivic repetition is seen to emerge from the constraints of the scheme. (We have also asked the computer to avoid parallel octaves and fifths.)


\begin{figure}
  \centering
  \hfill
  \includegraphics[width=\linewidth]{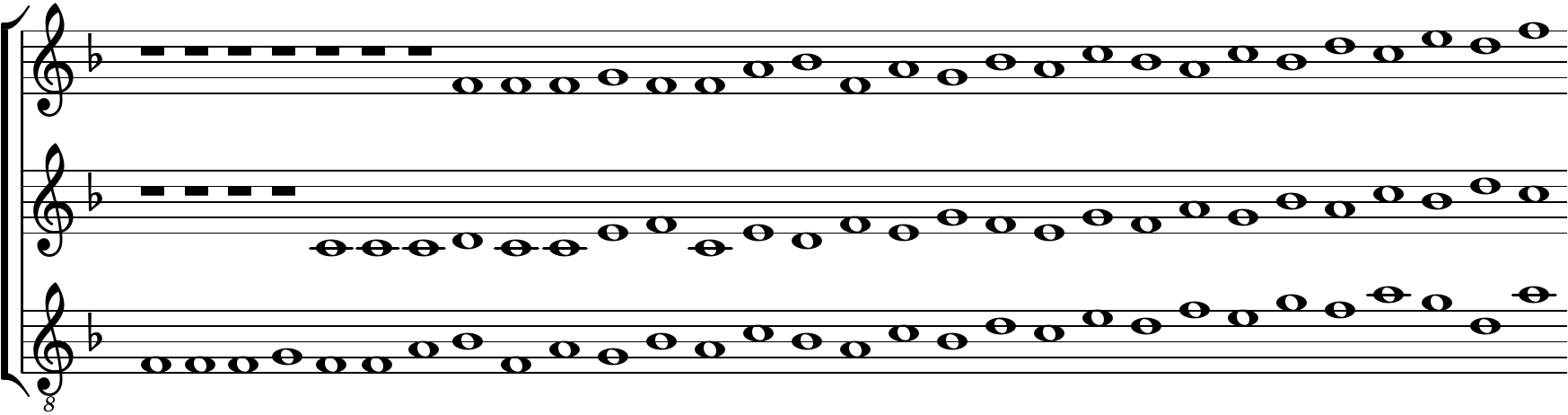}
  \caption{A randomly generated canon on the scheme $\S_{\mathrm{SC}}$ (see text).}
  \label{fig:SC_gen}
\end{figure}

\begin{figure}[p]
  \includegraphics[width=\linewidth]{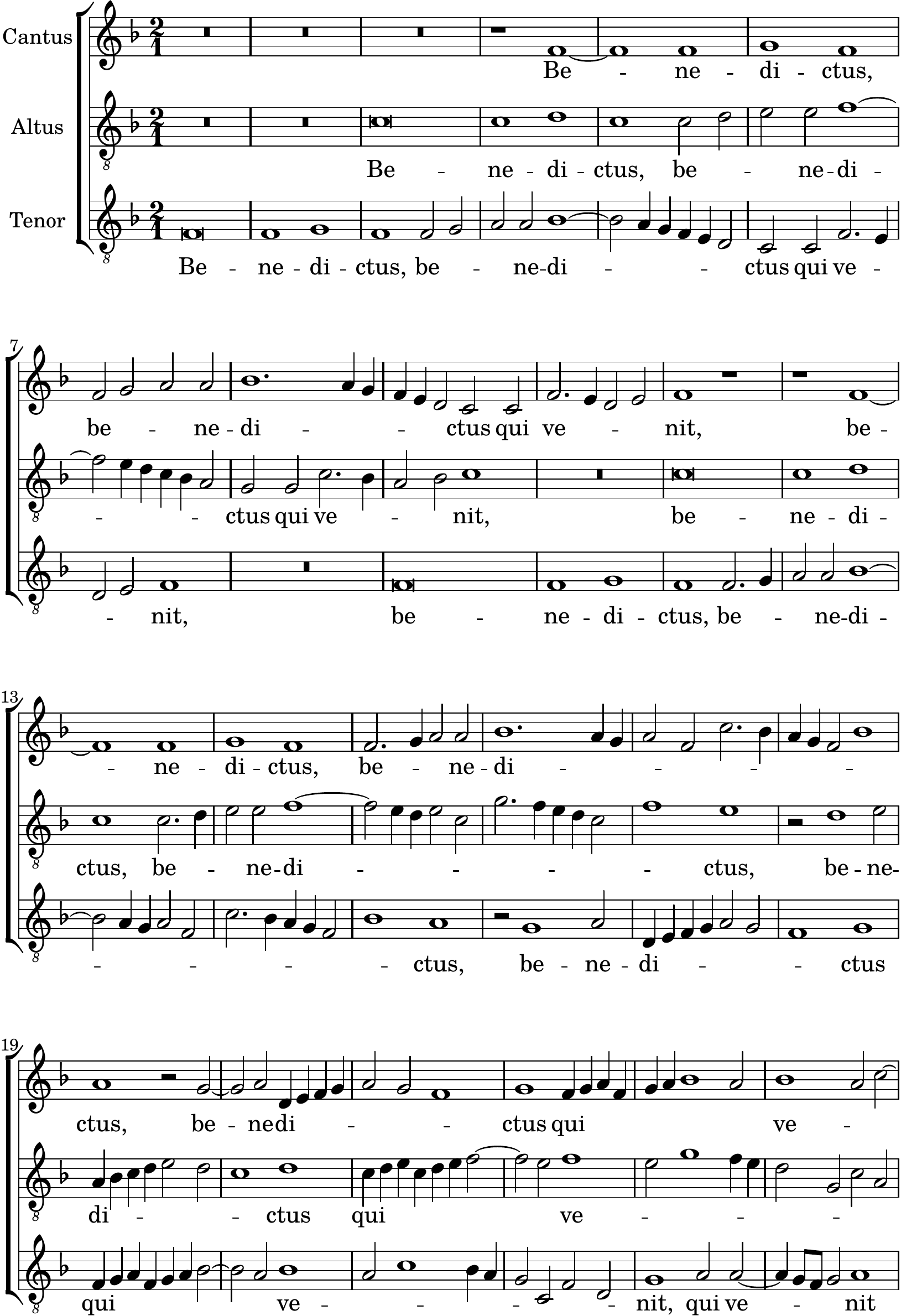}
  \caption{An original canon in Renaissance style, based on Palestrina's motet ``Sicut cervus,'' with a novel asymmetric scheme.}
  \label{fig:sc-begin}
\end{figure}
\begin{figure}[thbp]
  \includegraphics[width=\linewidth]{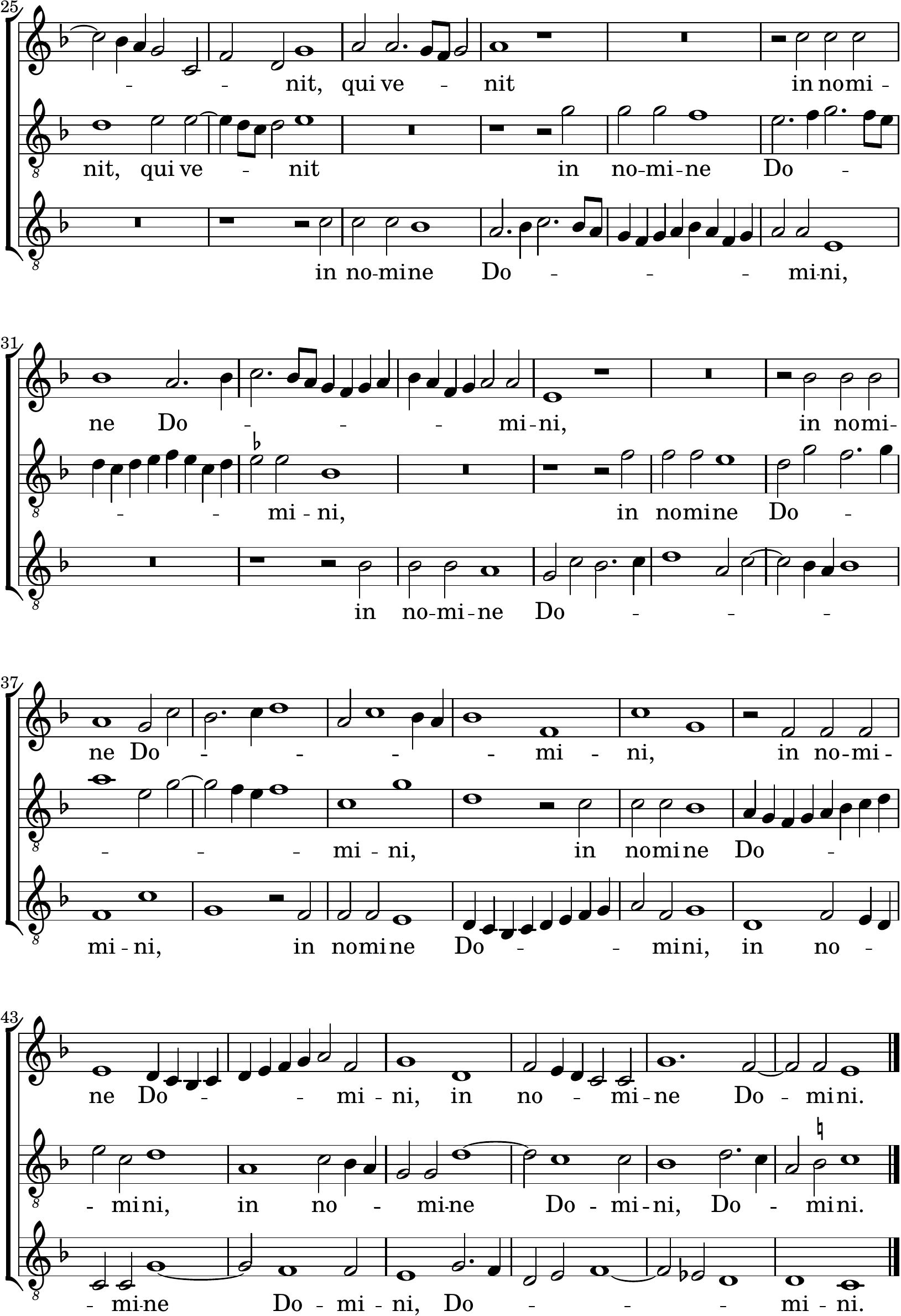}
  \caption{An original canon in Renaissance style, based on Palestrina's motet ``Sicut cervus,'' with a novel asymmetric scheme (continued).}
  \label{fig:sc-end}
\end{figure}

\section{Conclusion}

Our formalized notion of \emph{flexibility} provides a useful way to compare canonic schemes in Renaissance style. We recover that catches and stacked canons justly have pride of place; they are not just conceptually simpler, but they actually offer the most possibilities of any three-voice scheme. We also uncover latent unevenness among the various types of inverting canons. In asymmetric canons, we find a large number of essentially different types that, apparently, are mostly workable.

Our work leaves open the question of to what extent the composers in the sixteenth century ``knew'' the discrepancies between possible schemes. With inverting canons, we can answer affirmatively. Invertible counterpoint is a stock topic in contrapuntal theory as far back as \cite[ch.~56]{Zarlino}, who gives samples of invertible counterpoint at the 12th and 10th. It would have been plain to Palestrina that, in Figure \ref{fig:GiaFuBenedictus}, it was necessary to compose the bottom two voices in invertible counterpoint at the 12th in order to work with the top voice. 

With asymmetric canons, by contrast, no compositional recipe has come down to us; and in view of the complicated algebraic integers that appear in the analysis even for our vastly simplified theoretical model, it seems unlikely that the theoreticians of the Renaissance could have devised one. It seems that the composers who dared to write asymmetric canons (which are far, far less numerous than stacked canons, especially if we include two-voice canons and catches at the unison among the latter) used their trained skill at imitative polyphony to craft workable schemes and themes, unaided by any general theoretical guide. They successfully avoided a few schemes which we have demonstrated to be unworkable; on the other hand, there remain a large number of unattempted schemes expected to work just as well as those in the repertoire.

It would be instructive to extend this research to more than three voices. Pierre de la Rue's mass \emph{O salutaris hostia} \citep*{LaRue} (not to be confused with his motet by the same name) consists mainly of 4-in-1 canons where the time and pitch relation between soprano and alto is replicated between tenor and bass. This slightly mitigates the complexity of what is still a formidable challenge. Computation shows that the schemes used for the three movements of the Kyrie have flexibility values in the $2.4\sim 2.7$ range. An even more ambitious endeavor in this direction is Mathurin Forestier's \emph{Missa L'homme arm\'e} \citep*{Forestier}, sometimes attributed to Mouton (see  \cite*{BurnFurther}). It makes use of numerous two- and three-voice canons on the well-known theme ``L'homme arm\'e'' with systematically varying pitch intervals and ends with a striking canon for seven voices with the scheme
\[
  \{ (0,0)\tB, (1,3), (2,6), (3,0), (4,3), (5,6), (6,0) \}
\]
(time units are full measures of $3/1$). The scheme has a flexibility value of $2.420$, remarkably high for so many voices. To cite an even more extreme example, Gascongne's \emph{Ista est speciosa} is a stacked canon at the second for as many as twelve voices. All of these works from the Josquin generation contain musical effects that were already disfavored and would be unambiguously considered blemishes within a few decades: the La Rue has unsyncopated accented dissonances; the Forestier has a few parallel octaves and frequent rests (even in the final cadence); and the Gascongne, depending on the chromatic inflections chosen in the realization, is replete with either melodic tritones or chromatic modulations ending in a remote key \citep{BurnFurther}. These concessions show that composers were aware of the clashing demands of the canon and the style, and ultimately they decided in favor of what we find in Palestrina: canons of mostly simple types, executed in an exquisitely refined style. The findings of this paper call into question the idea that the dazzling works of the Josquin generation attain the limits of possibility within the style. By carefully choosing a scheme and diligently working out a melody, perhaps assisted by the many advances in computing power and AI since \cite{Lewin}, it may be possible to compose canons that combine the impressive devices of the Josquin generation with the sensitive contrapuntal control of Palestrina and his contemporaries.

There is no need to restrict one's sights to the Renaissance. For instance, Bach and Mozart wrote stacked canons for four voices. In principle, the theoretical method used in this paper can be applied to any polyphonic style in which certain vertical sonorities are preferred. The challenge of constructing a canon, and the added difficulties posed by asymmetric schemes, offer puzzles that will be appreciated by composers of all eras, especially those of a mathematical bent.

\appendix


\section{Technical proofs} \label{sec:technical}

\subsection{Proof of Proposition \ref{prop:eqv}} \label{sec:propeqv}
In this section, we prove the following:
\propeqv*
\begin{proof}
  Items \ref{it:ttrans}--\ref{it:inv} are proved by producing a bijection between the valid canons $(x_i)$ of the two schemes $\S$, $\S'$ in question, proving that $V_n(\S) = V_n(\S')$. For items \ref{it:ttrans} and \ref{it:ptrans}, the same melody $(x_i)$ serves for both. For item \ref{it:shear}, we use a shearing transformation $x_i' = x_i + ci$ to transform a valid canon $(x_i)$ for $\S$ to a valid canon $(x_i')$ for the new scheme $\S' = \{(t_i, p_i + ct_i)\}$. The vertical intervals $y_{ti} - y_{tj}$ are thereby preserved. For item \ref{it:inv}, we invert the melody by setting $x_i' = -x_i$. Note that if the $i$th voice is the marked bass voice, the transformation $p_i' = 5 - p_i$ implies that intervals above the bass get inverted at the 10th and so remain consonant.
  
  Item \ref{it:tdil} is interesting because it cannot be proved by simply finding a bijection; the numbers of canons $V_n(\S)$, $V_n(\S')$ may differ in general. For instance, if $\S'$ is derived by dilating the times $t_i$ by a large integer $c$, then $V_n(\S') = 7^n$ for all $n \leq c$ because there is no opportunity for two notes to sound simultaneously under the scheme $\S'$.
  
  Assume first that $c > 0$ is an integer. Note that a valid canon for $\S'$ is composed of $c$ interwoven valid canons for $\S$: a melody $x_1, x_2, \ldots, x_n$ is valid for $\S'$ if and only if the subsequences
  \begin{equation} \label{eq:subseq}
    x_a, x_{c + a}, x_{2c + a}, \ldots, \quad 1 \leq a \leq c
  \end{equation}
  are valid canons for $\S$. Now we use the division algorithm (dividing $n$ by $c$ with remainder) to write $n = qc + \ell$, where $0 \leq \ell < c$. We note that the sequence \eqref{eq:subseq} stops at $x_{qc + a}$ if $a \leq \ell$, at $x_{(q-1)c + a}$ otherwise. So
  \begin{equation} \label{eq:tdil}
    V_n(\S') = V_{qc + \ell}(\S') = V_q(\S)^{c - \ell} V_{q+1}(\S)^\ell.
  \end{equation}
  If $V_n(\S) = 0$ for all sufficiently large $n$, then the same is true for $\S'$, so both schemes have flexibility $0$. So we will assume that $\lambda(\S) \geq 1$. By definition, for every $\epsilon$ ($0 < \epsilon < 1$) there is an $N$ such that for all $n \geq N$,
  \[
  \big(\lambda(\S) - \epsilon\big)^n \leq V_n(\S) \leq \big(\lambda(\S) + \epsilon\big)^n.
  \]
  If $q > N$, then by \eqref{eq:tdil}, we get the upper bound
  \begin{align*}
    V_n(\S') &= V_q(\S)^{c - \ell} V_{q+1}(\S)^\ell \\
    &\leq \big(\lambda(\S) + \epsilon\big)^{q(c - \ell)} \cdot \big(\lambda(\S) + \epsilon\big)^{(q+1)\ell} \\
    &= \big(\lambda(\S) + \epsilon\big)^{qc + \ell} \\
    &= \big(\lambda(\S) + \epsilon\big)^n
  \end{align*}
  and, analogously, the lower bound
  \[
  V_n(\S') \geq \big(\lambda(\S) - \epsilon\big)^n.
  \]
  Thus, for $n \geq cN$, we have the desired bounds
  \[
  \lambda(\S) - \epsilon \leq \sqrt[n]{V_n(\S')} \leq \lambda(\S) + \epsilon.
  \]
  Since $\epsilon$ was arbitrary, we have shown that $\lambda(\S') = \lambda(\S)$.
  
  This concludes the proof in the case that $c$ is a positive integer. The other cases follow readily: if $c = a/b$ is a ratio of positive integers, we first multiply the time displacements $t_i$ by $a$ and then divide by $b$, using the previous case to show that both steps preserve the flexibility value. Lastly, if $c = -1$, we put the melody in retrograde (that is, $x_i' = x_{n+1-i}$), taking care of all the negative values of $c$.
\end{proof}

\subsection{Proof of Lemma \ref{lem:eig}} \label{sec:lemeig}

In this section we prove the following:
\lemeig*

\begin{proof}
 It suffices to demonstrate bounds of the form
\begin{equation}
  m \cdot \lambda^n \leq V_n(\S) \leq M \cdot n^d \cdot \lambda^n
\end{equation}
for some constants $M > m > 0$ and $d \geq 0$, as then the desired root $\sqrt[n]{V_n(\S)}$ will be squeezed between two quantities that tend to $\lambda$ as $n \to \infty$.

For the lower bound, let $\mathbf{v}_i$ be the nonnegative eigenvector of $A_i$, normalized so that its entries sum to $1$. Then, using formula \ref{eq:An} for $V_{n}(\S)$ in terms of the matrix $A_i$,
\begin{align*}
  V_{n + s}(\S) &= \1_N^{\top} A_\S^n \1_N \\
  & \geq \1_{N_i}^\top A_i^n \1_{N_1} \\
  & \geq \mathbf{v}_i^\top A_i^n \mathbf{v}_i \\
  & = \mathbf{v}_i^\top \lambda^n \mathbf{v}_i \\
  &= \lambda^n,
\end{align*}
establishing the lower bound with $m = \lambda^{-s}$.

As to the upper bound, we first claim that the terms $V_n(\S)$ satisfy the homogeneous linear recurrence
\[
\sum_{i = 0}^N c_i V_{n + i}(\S) = 0, \quad n \geq s + 1
\]
with characteristic polynomial $f$. This holds because
\begin{align*}
  \sum_{i = 0}^N c_i V_{n + i}(\S) &= \sum_{i = 0}^N c_i \mathbf{v}_i^\top A_\S^{n-s+i} \mathbf{v}_i \\
  &= \mathbf{v}_i^\top A_\S^{n-s}\left(\sum_{i=0}^n c_i A^i \right) \mathbf{v}_i \\
  &= \mathbf{v}_i^\top A_\S^{n-s} f(A) \mathbf{v}_i,
\end{align*}
and $f(A) = 0$ by the Cayley-Hamilton theorem. Now, applying the standard theorem on such linear recurrences, we deduce that, if $f$ has roots $\lambda = \mu_1, \mu_2, \ldots, \mu_\ell$ of multiplicities $d_1,\ldots, d_\ell$, then there exist polynomials $g_1, \ldots, g_\ell$ with degrees $\deg(g_i) \leq d_i - 1$ such that the formula
\[
V_n(\S) = \sum_{i = 1}^\ell g_i(n) \mu_i^n
\]
holds for all sufficiently large $n$ (specifically, for $n \geq s + 1 + d_0$ where $d_0$ is the multiplicity of $0$ as a root of $f$). But all $\lvert\mu_i\rvert < \lambda$, so letting $d = \max\{d_i\} - 1$, we get a bound of the form $V_n \leq M \cdot n^d \cdot \lambda^n$, as desired.
\end{proof}

\section{Table of flexibility values for three-voice canons}
\label{app:tab}

In Table \ref{tab:flex}, we list the flexibility values of all three-voice canons with time displacements $0 = t_1 < t_2 < t_3 \leq 8$, with the reductions of the number of cases explained in Section \ref{sec:data}.
\begin{table}
  \begin{tabular}{@{ }c@{ }c@{ }c|l@{}l@{}l@{}l@{}l@{}l@{}l@{\!\!}}
    $t_2$ & $t_3$ & bass & $p_3 = 0$ & $1$ & $2$ & $3$ & $4$ & $5$ & $6$ \\ \hline
    1 & 2 & none & 3.935\ftn{ti:stacked} & 2     & 3.562 & 3\ftn{ti:inv}    & 3     & 3.562\ftn{ti:For1} & 2     \\
    &   &  1st & 3.140\ftn{ti:stacked-u} & 2     & 2.732 & 2.732 & 2     & 3.140 & 1\ftn{ti:impos}     \\
    &   &  2nd & 3\ftn{ti:inv8}     & 1     & 3.562\ftn{ti:inv10} & 1     & 3\uu{\ref*{ti:inv}}     & 2     & 2     \\
    &   &  3rd & 3.140 & 2     & 2.732 & 2.732 & 2\ftn{ti:Hercules}     & 3.140 & 1     \\
    1 & 3 & none & 3.542\ftn{ti:For2} & 2.754 & 3.438 & 3.169 & 3.169 & 3.438 & 2.754\ftn{ti:reple} \\
    &   &  1st & 2.896 & 2.622 & 2.622 & 2.896 & 2.206 & \textbf{2.992\ftn{ti:fugam-pleni}} & 2.206 \\
    &   &  2nd & 2.754 & 2.247 & 3.033 & \textbf{1.618\ftn{ti:fib}} & 3.033 & 2.247 & 2.754 \\
    &   &  3rd & 2.945 & 2.481 & 2.781\ftn{ti:non-nobis} & 2.781 & 2.481 & 2.945 & 2.081 \\
    1 & 4 & none & 3.436 & 3.003 & 3.383 & 3.235 & 3.235 & 3.383 & 3.003 \\
    &   &  1st & 2.798 & 2.798 & 2.592 & 2.904 & 2.300 & 2.904 & 2.592 \\
    &   &  2nd & 2.666 & 2.666 & 2.875 & 2.346 & 2.952 & 2.346 & 2.875 \\
    &   &  3rd & 2.886 & 2.639 & 2.801 & 2.801 & 2.639 & 2.886 & 2.413 \\
    1 & 5 & none & 3.390 & 3.118 & 3.358 & 3.266 & 3.266 & 3.358 & 3.118\ftn{ti:accepit} \\
    &   &  1st & 2.747 & 2.849 & 2.583 & 2.883 & 2.583 & 2.849 & 2.747 \\
    &   &  2nd & 2.634 & 2.801 & 2.801 & 2.634 & 2.888 & 2.399 & 2.888 \\
    &   &  3rd & 2.861 & 2.709 & 2.809 & 2.809 & 2.709 & 2.861 & 2.565 \\
    2 & 5 & none & 3.398 & 3.363\ftn{ti:virgo} & 3.260 & 3.098 & 3.098 & 3.260 & 3.363 \\
    &   &  1st & 2.806 & 2.676 & 2.489 & 2.676 & 2.806 & 2.873 & 2.873 \\
    &   &  2nd & 2.595 & 2.795 & 2.902 & 2.902 & 2.795 & 2.595 & 2.325 \\
    &   &  3rd & 2.845 & 2.755 & 2.611 & 2.611 & 2.755 & 2.845 & 2.876 \\
    1 & 6 & none & 3.367 & 3.181 & 3.345 & 3.282 & 3.282 & 3.345 & 3.181 \\
    &   &  1st & 2.718 & 2.858 & 2.581 & 2.858 & 2.718 & 2.811 & 2.811 \\
    &   &  2nd & 2.620 & 2.844 & 2.758 & 2.758 & 2.844 & 2.620 & 2.873 \\
    &   &  3rd & 2.848 & 2.745 & 2.813 & 2.813 & 2.745 & 2.848 & 2.646 \\
    1 & 7 & none & 3.353 & 3.218 & 3.337 & 3.292 & 3.292\ftn{ti:fugam-agnus} & 3.337 & 3.218 \\
    &   &  1st & 2.698 & 2.853 & 2.698 & 2.836 & 2.784 & 2.784 & 2.836 \\
    &   &  2nd & 2.614 & 2.853 & 2.731 & 2.812 & 2.812 & 2.731 & 2.853 \\
    &   &  3rd & 2.841 & 2.767 & 2.816 & 2.816 & 2.767 & 2.841 & 2.695 \\
    2 & 7 & none & 3.356 & 3.340 & 3.290 & 3.208 & 3.208 & 3.290 & 3.340\ftn{ti:sine} \\
    &   &  1st & 2.745 & 2.643 & 2.745 & 2.814 & 2.849 & 2.849 & 2.814 \\
    &   &  2nd & 2.578 & 2.715 & 2.810 & 2.859 & 2.859 & 2.810 & 2.715 \\
    &   &  3rd & 2.833 & 2.791 & 2.722 & 2.722 & 2.791 & 2.833 & 2.847 \\
    3 & 7 & none & 3.359 & 3.288 & 3.203 & 3.341 & 3.341 & 3.203 & 3.288 \\
    &   &  1st & 2.785 & 2.852 & 2.785 & 2.702 & 2.835 & 2.835 & 2.702 \\
    &   &  2nd & 2.555 & 2.808 & 2.863 & 2.703 & 2.703 & 2.863 & 2.808 \\
    &   &  3rd & \textbf{2.814\ftn{ti:new}} & 2.846 & 2.753 & 2.753 & 2.846 & 2.814 & 2.662 \\
    1 & 8 & none & 3.344\ftn{ti:primi} & 3.242 & 3.333 & 3.298 & 3.298 & 3.333 & 3.242 \\
    &   &  1st & 2.685 & 2.842 & 2.763 & 2.816 & 2.816 & 2.763 & 2.842 \\
    &   &  2nd & 2.713 & 2.849 & 2.713 & 2.834 & 2.788 & 2.788 & 2.834 \\
    &   &  3rd & 2.836 & 2.780 & 2.818 & 2.818 & 2.780 & 2.836 & 2.726 \\
    3 & 8 & none & 3.348 & 3.295 & 3.230 & 3.335 & 3.335 & 3.230 & 3.295 \\
    &   &  1st & 2.764 & 2.842 & 2.816 & 2.687 & 2.816 & 2.842 & 2.764 \\
    &   &  2nd & 2.681 & 2.837 & 2.837 & 2.681 & 2.777 & 2.858 & 2.777 \\
    &   &  3rd & 2.816 & 2.840 & 2.771 & 2.771 & 2.840 & 2.816 & 2.703
  \end{tabular}~%
  \begin{tabular}{@{}r|p{1.1in}}
    \ref*{ti:stacked} & Acc.\ stacked canons and catches \\
    \ref*{ti:inv} & $\vee$-canons at the 12th \\
    \ref*{ti:For1} & Forestier's mass \emph{L'homme arm\'e,} Agnus Dei I \\
    \ref*{ti:stacked-u} & Unacc.\ stacked canons and catches; $\wedge$-canons at the octave \\
    \ref*{ti:impos} & Inflexible canon (see Figure \ref{fig:impos}) \\
    \ref*{ti:inv8} & Unacc.\ $\vee$-canons at the octave \\
    \ref*{ti:inv10} & Unacc.\ $\vee$-canons at the 10th \\
    \ref*{ti:Hercules} & Unacc.\ $\wedge$-canons at the 12th \\
    \ref*{ti:For2} & Forestier's mass \emph{L'homme arm\'e,} Agnus Dei II \\
    \ref*{ti:reple} & \emph{Repleatur,} Agnus Dei II \\
    \ref*{ti:fugam-pleni} & \textbf{\emph{Ad fugam,} Pleni} \\
    \ref*{ti:fib} & \textbf{Fibonacci scheme (see Sections \ref{sec:fib} and \ref{sec:new})} \\
    \ref*{ti:non-nobis} & ``Non nobis Domine'' \\
    \ref*{ti:accepit} & ``Accepit Jesus calicem'' \\
    \ref*{ti:virgo} & ``Virgo prudentissima'' (I pars) \\
    \ref*{ti:fugam-agnus} & \emph{Ad fugam,} Agnus Dei II \\
    \ref*{ti:sine} & \emph{Sine nomine,} Agnus Dei II \\
    \ref*{ti:new} & \textbf{New canon (see~Section~\ref{sec:new})} \\
    \ref*{ti:primi} & \emph{Primi toni,} Agnus Dei II
  \end{tabular}
  \label{tab:flex}
  \vspace{1pt}
  \caption{The flexibility values of the $3$-voice canonic schemes $\{(0,0), (t_2,0), (t_3,p_3)\}$ with $t_3$ up to $8$. The footnotes indicate schemes that occur in the repertoire.}
\end{table}

\section*{Acknowledgements}
\addcontentsline{toc}{section}{Acknowledgements}
I thank Peter Smith for helpful guidance. I thank the three anonymous reviewers for patiently reading an earlier draft and catching many typos.

\section*{Supplemental online material}
\addcontentsline{toc}{section}{Supplemental online material}

Supplemental online material for this article can be accessed at \url{doi-provided-by-publisher} and/or \url{https://github.com/emo916math/Renaissance-canons}. The Online Supplement includes the following:
\begin{itemize}
  \item MIDI realizations of the musical examples in this paper. To clarify the counterpoint and voice crossing, a different MIDI instrument has been used for each voice in the texture.
  \item A Sage program that was used to compute the flexibility values in this paper.
\end{itemize}


\section*{Disclosure statement}
\addcontentsline{toc}{section}{Disclosure statement}

No potential conflict of interest was reported by the authors.

\section*{ORCID}
\addcontentsline{toc}{section}{ORCID}

\bibliographystyle{tMAM}
\bibliography{Music}

\begin{thebibliography}{20}
\newcommand{\enquote}[1]{``#1''}
\providecommand{\natexlab}[1]{#1}
\providecommand{\url}[1]{\normalfont{#1}}
\providecommand{\urlprefix}{ }
\expandafter\ifx\csname urlstyle\endcsname\relax
  \providecommand{\doi}[1]{doi:\discretionary{}{}{}#1}\else
  \providecommand{\doi}{doi:\discretionary{}{}{}\begingroup
  \urlstyle{rm}\Url}\fi

\bibitem[Arias-Valero, Agustín-Aquino, and Lluis-Puebla(2022)]{Arias}
Arias-Valero, Juan, Octavio Agustín-Aquino, and Emilio Lluis-Puebla. 2022.
\newblock ``Musicological, computational, and conceptual aspects of
  first-species counterpoint theory.'' \emph{Journal of Mathematics and Music}
  1--17.

\bibitem[Burn(2001)]{BurnFurther}
Burn, David. 2001.
\newblock ``Further Observations on Stacked Canon and {R}enaissance
  Compositional Procedure: {G}ascongne's `{I}sta Est Speciosa' and
  {F}orestier's `{M}issa L'Homme Armé'.'' \emph{Journal of Music Theory} 45
  (1): 73--118.
\newblock \urlprefix\url{http://www.jstor.org/stable/3090649}.

\bibitem[Farbood and Schoner(2001)]{Farbood}
Farbood, Morwaread, and Bernd Schoner. 2001.
\newblock ``Analysis and Synthesis of Palestrina-Style Counterpoint Using
  Markov Chains.'' \emph{International Computer Music Conference Proceedings}
  2001: 4 pp.

\bibitem[Feininger(1937)]{Feininger}
Feininger, Laurence K.~J. 1937.
\newblock ``Die {F}r\"uhgeschichte des {K}anons bis {J}osquin des {P}rez (um
  1500).'' Thesis ({P}h.{D}.), University of Heidelberg.

\bibitem[Forestier(1996)]{Forestier}
Forestier, Mathurin. 1996.
\newblock ``Missa L'Homme arm\'e.'' In \emph{Opera omnia,}   edited by T.~G.
  MacCracken and N.~S. Josephson, Corpus mensurabilis musicae, no.~104. S.l.:
  Neuhausen-Stuttgart: American Institute of Musicology; Hänssler-Verlag.

\bibitem[Gauldin(1996)]{GauldinComposition}
Gauldin, Robert. 1996.
\newblock ``The Composition of Late {R}enaissance Stretto Canons.''
  \emph{Theory and Practice} 21: 29--54.
\newblock \urlprefix\url{http://www.jstor.org/stable/41054290}.

\bibitem[Gauldin(2013)]{gauldin2013practical}
Gauldin, Robert. 2013.
\newblock \emph{A Practical Approach to 16th Century Counterpoint: Revised
  Edition}.
\newblock Waveland Press.
\newblock \urlprefix\url{https://books.google.com/books?id=afAVAAAAQBAJ}.

\bibitem[Gennaro~Auricchio and Lanzarotto(2023)]{tilings}
Gennaro~Auricchio, Luca~Ferrarini, and Greta Lanzarotto. 2023.
\newblock ``An integer linear programming model for tilings.'' \emph{Journal of
  Mathematics and Music} 17 (3): 514--530.
\newblock \urlprefix\url{https://doi.org/10.1080/17459737.2023.2180812}.

\bibitem[Gosman(1997)]{GosmanStacked}
Gosman, Alan. 1997.
\newblock ``Stacked Canon and {R}enaissance Compositional Procedure.''
  \emph{Journal of Music Theory} 41 (2): 289--317.
\newblock \urlprefix\url{http://www.jstor.org/stable/843961}.

\bibitem[Herremans and Sörensen(2013)]{Herremans}
Herremans, D., and K.~Sörensen. 2013.
\newblock ``Composing fifth species counterpoint music with a variable
  neighborhood search algorithm.'' \emph{Expert systems with applications} 40
  (16): 6427--6437.

\bibitem[Jeppesen(2012)]{jeppesen2012style}
Jeppesen, K. 2012.
\newblock \emph{The Style of Palestrina and the Dissonance}.
\newblock Dover Books on Music: Analysis. Dover Publications.
\newblock \urlprefix\url{https://books.google.com/books?id=GQ6TcQvYnZoC}.

\bibitem[Jeppesen, Haydon, and Mann(1992)]{jeppesen1992counterpoint}
Jeppesen, K., G.~Haydon, and A.~Mann. 1992.
\newblock \emph{Counterpoint: The Polyphonic Vocal Style of the Sixteenth
  Century}.
\newblock The Prentice-Hall music series. Dover Publications.
\newblock \urlprefix\url{https://books.google.com/books?id=OcSVGkug58gC}.

\bibitem[{Josquin des Prez}(1505/2002)]{Josquin}
{Josquin des Prez}. 1505/2002.
\newblock ``Missa Hercules dux Ferrariae.'' Amsterdam.
\newblock
  \urlprefix\url{https://imslp.org/wiki/Missa_Hercules_dux_Ferrariae%2C_NJE_11.1_(Josquin_Desprez)}.

\bibitem[La~Rue(1516/1996)]{LaRue}
La~Rue, Pierre~de. 1516/1996.
\newblock ``Missa O salutaris hostia.'' In \emph{Missa O gloriosa Margaretha;
  Missa O salutaris hostia; Missa Pascale; Missa Pro fidelibus defunctis; Missa
  Puer natus est nobis,}  Vol.~5  edited by Nigel St. John~Davison. American
  Institute of Musicology; Hänssler-Verlag.
\newblock
  \urlprefix\url{https://imslp.org/wiki/Missa_O_salutaris_hostia_(La_Rue%2C_Pierre_de)}.

\bibitem[Lewin(1983)]{Lewin}
Lewin, David~Benjamin. 1983.
\newblock ``An Interesting Global Rule for Species Counterpoint.'' \emph{In
  Theory Only} 6/8: 19--44.

\bibitem[Meyer(2000)]{MeyerMatrix}
Meyer, C.~D. 2000.
\newblock \emph{Matrix analysis and applied linear algebra}.
\newblock Philadelphia: Society for Industrial and Applied Mathematics.

\bibitem[Palestrina(1569--1601/1862--1907)]{Pal}
Palestrina, Giovanni Pierluigi~da. 1569--1601/1862--1907.
\newblock ``Opera omnia {I}oannis {P}etraloysii {P}raenestini.''
  \urlprefix\url{https://imslp.org/wiki/Opera_omnia_Ioannis_Petraloysii_Praenestini_(Palestrina,_Giovanni_Pierluigi_da)}.

\bibitem[Robinson and Hall(1989)]{Aldrich}
Robinson, B.W., and R.F. Hall, eds. 1989.
\newblock \emph{The Aldrich Book of Catches}.
\newblock Novello.
\newblock \urlprefix\url{https://books.google.com/books?id=dsE5AQAAIAAJ}.

\bibitem[van Geenen(2012)]{Geenen}
van Geenen, Jurjen~L. 2012.
\newblock ``On designing stacked canons with relative chord tones.''
  \emph{Journal of Mathematics and Music} 6 (3): 187--205.
\newblock \urlprefix\url{https://doi.org/10.1080/17459737.2012.744443}.

\bibitem[Zarlino(1558/1968)]{Zarlino}
Zarlino, Gioseffo. 1558/1968.
\newblock \emph{The art of counterpoint. {P}art three of {L}e istituzioni
  harmoniche}.
\newblock Music theory translation series [2]. New Haven: Yale University
  Press.

\end{thebibliography}

\addcontentsline{toc}{section}{References}

\end{document}